\documentclass[12pt]{article}
\usepackage{amssymb}
\usepackage{mathrsfs}
\usepackage[hypertex]{hyperref}
\usepackage{amsfonts}
\usepackage{graphicx}
\usepackage{mathptmx}
\usepackage{latexsym,amsmath,amssymb,amsfonts,amsthm}

\newtheorem{theorem}{Theorem}[section]
\newtheorem{lemma}[theorem]{Lemma}

\newtheorem{corollary}[theorem]{Corollary}

\newtheorem{remark}[theorem]{Remark}

\newtheorem{defn}[theorem]{Definition}

\begin{document}
\setcounter{page}{1}
\title{Geometry and topology of manifolds with
integral radial curvature bounds}
\author{Jing Mao}
\date{\emph{In memory of my father Mr. Xu-Gui Mao}}
\protect\footnotetext{\!\!\!\!\!\!\!\!\!\!\!\!{MSC 2010:} Primary
53C20.
}
\maketitle ~~~\\[-15mm]
\begin{center}{\footnotesize
Faculty of Mathematics and Statistics, \\
Key Laboratory of Applied Mathematics of Hubei Province, \\
Hubei University, Wuhan 430062, China}
\end{center}

\begin{abstract}
In this paper, we systematically investigate the geometry and
topology of manifolds with integral \emph{radial} curvature bounds,
and obtain many interesting and important conclusions.
\end{abstract}

\markright{\sl\hfill  J. Mao\hfill}

\section{Introduction}
\renewcommand{\thesection}{\arabic{section}}
\renewcommand{\theequation}{\thesection.\arabic{equation}}
\setcounter{equation}{0} \setcounter{maintheorem}{0}

As we know, \emph{curvature} is a fundamental concept in
Differential Geometry, and through which, we can describe
\emph{accurately} the differences between two manifolds with
different curvatures. There are so many examples to support this
viewpoint and readers do not need to spend even one second on
finding such an example. An intuitive example appearing in readers'
mind automatically might be the classical Bishop's volume comparison
theorems, which tell us that for a disk in the Euclidean $2$-space
$\mathbb{R}^2$, one can increase (or decrease) its volume (i.e.,
$2$-dimensional Hausdorff measure) by decreasing (or increasing) its
zero Gaussian curvature. This change of Gaussian curvature naturally
leads to the change of shape of the Euclidean $2$-disk. Speaking in
other more accurate words, the geometry (or topology) of the
Euclidean $2$-disk changes during this increasing (or decreasing)
process of the volume. Of course, the statement of the Bishop's
volume comparisons (which also works for higher dimension cases) is
more complicated, but this interesting application already reveals
the importance of curvature.

In this paper, we introduce a concept named ``\emph{integral radial
(Ricci or sectional) curvatures}", which in essence is the
$L^{p}$-norm of the part of radial (Ricci or sectional) curvature
not greater (or not less) than a prescribed continuous function of
the Riemannian distance parameter - see Definitions \ref{def1-1} and
\ref{def1-2}, Remark \ref{remark2-1} for details.

 It is well-known
that on any surface $\mathcal{S}^2$, the Gauss-Bonnet formula says
that the Euler characteristic number $\chi(\mathcal{S}^2)$ is given
by $ \chi(\mathcal{S}^2)=\int_{\mathcal{S}^2}K(g)dv $ for any metric
$g$ on $\mathcal{S}^2$, with $K(g)$ the Gaussian curvature of
$\mathcal{S}^2$. For the case of higher dimensions (i.e., for any
$n$-dimensional manifold $\mathcal{S}^n$), the Chern-Weil formulae
for characteristic classes are given by the integral of some
polynomial of degree $\frac{n}{2}$ in the curvature. These two facts
tell us that one might estimate topological invariants on a
prescribed manifold by using the average of curvatures of any metric
on this manifold. This is exactly our motivation of investigating
integral norms of curvatures.

We systematically investigate the geometry and topology of manifolds
with integral radial curvature bounds, and fortunately, we can
obtain the followings:

\begin{itemize}

\item  For a given complete $n$-Riemannian ($n\geq2$) manifold $M$, in Section \ref{VC}, several upper bounds, involving integral
radial (Ricci or sectional) curvatures, for the volumes of geodesic
balls, geodesic cones, normal tubes (around a prescribed
submanifold) on $M$ have been shown - see Theorems \ref{theorem2-1},
\ref{theorem3-8},  \ref{theorem3-9},  \ref{theorem3-10},
\ref{theorem3-13} and \ref{theorem3-15} for details. It is not hard
to see that these upper bound estimates cover those shown
respectively by S. Gallot, P. Peterson, S.-D. Shteingold, G.-F. Wei,
D. Yang in  \cite[Theorems 1 and 2]{sg}, \cite[Theorems 2.4 and
2.5]{pw11}, \cite[Theorem1.1]{pw1} and
 \cite[Theorem 7.1]{yd1} as special
cases. Besides, these volume estimates can also give some
byproducts. For instance, as explained clearly in (3) of Remark
\ref{remark3-2}, if the Type-I integral radial Ricci curvature (see
Definition \ref{def1-1}) vanishes identically, then Theorem
\ref{theorem2-1} directly gives the Bishop-Gromov type relative
volume comparison estimate and the Bishop-type volume comparison
(proven in \cite{fmi,m1,m3}) for manifolds having a radial Ricci
curvature lower bound.

\item Applying our volume estimate for geodesic cones (see Theorem
\ref{theorem3-10}), if the Type-I integral radial Ricci curvature
was assumed to be bounded from above, we can give lower bounds for
the local isoperimetric constant and the local Sobolev constant of
geodesic balls -- see Theorem \ref{theorem4-1} and (1) of Remark
\ref{remark4-2} for details. This fact is a generalization of D.
Yang's lower bound estimate given in \cite[Theorem 7.4]{yd1}. By
defining an isoperimetric quantity $\mathrm{Is}(p)$ (see
(\ref{IQ-D})) and applying  \cite[Lemma 4]{sg}, our volume estimate
for normal tubes around hypersurfaces (see Theorem
\ref{theorem3-13}), an interesting isoperimetric inequality can be
obtained (see Theorem \ref{theorem4-4} for details), which can be
seen as an extension of S. Gallot's result \cite[Theorem 3]{sg}. By
mainly using Theorem \ref{theorem3-15} and the variational
principle, we can give a sharp upper bound for the infimum of the
spectrum of the Laplacian $\Delta$ on complete noncompact manifolds
(see Theorem \ref{theorem4-5}), which improves H. Donnelly's and S.
Gallot's estimates shown separately in \cite{hdo,sg}. A nice sharp
upper bound can also be given for the infimum of the spectrum of the
nonlinear $\flat$-Laplacian $\Delta_{\flat}$, $1<\flat<\infty$, on
complete noncompact manifolds -- see Corollary \ref{corollary4-7}
for details.

\item As a direct consequence of our volume doubling result (see Corollary
\ref{corollary3-6}), if the average of the Type-I integral radial
Ricci curvature (see also Definition \ref{def1-1}) is sufficiently
small and $\ell(q)$, defined by (\ref{key-def1}), has an upper
bound, an interesting compactness conclusion for a collection of
closed Riemannian $n$-manifolds ($n\geq2$) can be obtained -- see
Corollary \ref{corollary5-1} for details. This result is an
extension of \cite[Corollary 1.3]{pw1} given by P. Peterson and
G.-F. Wei. Besides, if the Type-II integral radial sectional
curvature (see Definition \ref{def1-2}) on a given closed manifold
is sufficiently small, then using the volume estimate for normal
tubes around geodesics (see Theorem \ref{theorem3-9}), we can give a
positive lower bound for the length of the shortest closed geodesic
on this closed manifold (see Theorem \ref{theorem5-2} for details),
which generalizes J. Cheeger's related conclusion in \cite{c1}.
Finally, if the Type-I integral radial Ricci curvature was assumed
to be bounded from above, then using our volume comparison (see
Theorem \ref{theorem2-1}),
 a Buser-type isoperimetric inequality can be
obtained, which \emph{partially} extends P. Buser's classical result
in \cite{pb} and S.-H. Paeng's conclusion \cite[Theorem 1.2]{shp} a
lot.

\end{itemize}

At the end of this paper, we also issue open problems, which are
worth investigating in the coming future.

\section{Preliminaries} \label{S2-P}
\renewcommand{\thesection}{\arabic{section}}
\renewcommand{\theequation}{\thesection.\arabic{equation}}
\setcounter{equation}{0} \setcounter{maintheorem}{0}

In this section, we would like to give the concept, \emph{integral
radial (Ricci or sectional) curvatures}, in detail. However, in
order to state clearly,  first we prefer to give some preliminaries,
which have been introduced by the author in some of his previous
articles (see, e.g., \cite{fmi,m1,m2,m3,mdw,ywmd}).

Let $(M,g)$ be a  complete Riemannian $n$-manifold ($n\geq2$) with
the metric $g$, and $\nabla$ be the gradient operator. For a point
$q\in M$, one can set up a geodesic polar coordinates $(t,\xi)$
around this point $q$, where $\xi\in{S}_{q}^{n-1}\subseteq{T_{q}M}$
is a unit vector of the unit sphere $S_{q}^{n-1}$ with center $q$ in
the tangent space $T_{q}M$. Let $\mathcal{D}_{q}$ and $d_{\xi}$ be
defined by
\begin{eqnarray*}
\mathcal{D}_{q}=\{t\xi|~0\leq{t}<d_{\xi},~\xi\in{S^{n-1}_{q}}\},
\end{eqnarray*}
and
\begin{eqnarray*}
d_{\xi}=d_{\xi}(q):=  \sup\{t>0|~\gamma_{\xi}(s):= \exp_q(s\xi)~
{\rm{is~ the~ unique~ minimal~ geodesic ~joining} }~ q ~{\rm{and}}
~\gamma_{\xi}(t)\}
\end{eqnarray*}
respectively. Then $\exp_q:\mathcal{D}_q \to M\backslash Cut(q)$
gives a diffeomorphism from $\mathcal{D}_q$ onto the open set
$M\backslash Cut(q)$, with $Cut(q)$ the cut locus of  $q$. For
$\zeta\in{\xi^{\bot}}$, one can define the path of linear
transformations $\mathbb{A}(t,\xi):\xi^\perp\rightarrow{\xi^\perp}$
as follows
 \begin{eqnarray*}
\mathbb{A}(t,\xi)\zeta=(\tau_{t})^{-1}Y(t),
 \end{eqnarray*}
with $\xi^\perp$ the orthogonal complement of $\{\mathbb{R}\xi\}$ in
$T_{q}M$, where $\tau_{t}:T_{q}M\rightarrow{T_{\exp_{q}(t\xi)}M}$ is
the parallel translation along the geodesic $\gamma_{\xi}(t)$ with
$\gamma'(0)=\xi$, and $Y(t)$ is the Jacobi field along
$\gamma_{\xi}$ satisfying $Y(0)=0$, $(\triangledown_{t}Y)(0)=\zeta$.
Set
\begin{eqnarray*}
\mathcal{R}(t)\zeta=(\tau_{t})^{-1}R(\gamma'_{\xi}(t),\tau_{t}\zeta)
\gamma'_{\xi}(t),
\end{eqnarray*}
where the curvature tensor $R(X,Y)Z$ is defined by
$R(X,Y)Z=-[\nabla_{X},$ $ \nabla_{Y}]Z+ \nabla_{[X,Y]}Z$. Then
$\mathcal{R}(t)$ is a self-adjoint operator on $\xi^{\bot}$, whose
trace is the radial Ricci tensor
$$\mathrm{Ric}_{\gamma_{\xi}(t)}(\gamma'_{\xi}(t),\gamma'_{\xi}(t)).$$
Clearly, the map $\mathbb{A}(t,\xi)$ satisfies the Jacobi equation
 $\mathbb{A}''+\mathcal{R}\mathbb{A}=0$ with initial conditions
 $\mathbb{A}(0,\xi)=0$, $\mathbb{A}'(0,\xi)=I$, and  by Gauss's lemma, the Riemannian metric of $M$ can be
 expressed by
 \begin{eqnarray} \label{2.1}
 ds^{2}(\exp_{q}(t\xi))=dt^{2}+\|\mathbb{A}(t,\xi)d\xi\|^{2}
 \end{eqnarray}
on the set $\exp_{q}(\mathcal{D}_q)$. Consider the metric components
$g_{ij}(t,\xi)$, $i,j\geq 1$, in a coordinate system $\{t, \xi_a\}$
formed by fixing  an orthonormal basis $\{\zeta_a, a\geq 2\}$ of
 $\xi^{\bot}=T_{\xi}S^{n-1}_q$, and extending it to a local frame $\{\xi_a, a\geq2\}$ of
$S_q^{n-1}$. On $\mathcal{D}_{q}$, one can define a function $J>0$
as follows
\begin{eqnarray} \label{function-J}
J^{n-1}(t,\xi)=\det\mathbb{A}(t,\xi)=\sqrt{\|g\|}:=\sqrt{\det[g_{ij}]}
\end{eqnarray}
 Since
 $\tau_t: S_q^{n-1}\to S_{\gamma_{\xi}(t)}^{n-1}$ is an
isometry, we have
$$
g\left( d(\exp_q)_{t\xi}(t\zeta_{a}),
d(\exp_q)_{t\xi}(t\zeta_{b})\right)=g\left(
\mathbb{A}(t,\xi)(\zeta_{a}), \mathbb{A}(t,\xi)(\zeta_{b})\right),
$$
and
$$ \sqrt{\|g\|}=\det\mathbb{A}(t,\xi).$$
So, by (\ref{2.1}), the volume $\mathrm{vol}(B(q,r))$ of the
geodesic ball $B(q,r)$ on $M$ is given by
\begin{eqnarray} \label{2.2}
\mathrm{vol}(B(q,r))=\int_{S_{q}^{n-1}}\int_{0}^{\min\{r,d_{\xi}\}}\sqrt{\|g\|}dtd\sigma=\int_{S_{q}^{n-1}}\left(\int_{0}^{\min\{r,d_{\xi}\}}\det(\mathbb{A}(t,\xi))dt\right)d\sigma,
\end{eqnarray}
where $d\sigma$ stands for the $(n-1)$-dimensional volume element on
$\mathbb{S}^{n-1}\equiv S_{q}^{n-1}\subseteq{T_{q}M}$. \emph{In the
sequel, we make an agreement that $\mathrm{vol}(\cdot)$ denotes the
volume of the prescribed geometric object under the related
Hausdorff measure}. Let
 \begin{eqnarray} \label{inj-R}
\mathrm{inj}(q):=d_{M}(q,Cut(q))=\min\limits_{\xi\in S_{q}M}d_{\xi}
 \end{eqnarray}
 be the injectivity radius at $q$.
 In general, we have
$B(q,\mathrm{inj}(q))\subseteq{M}\backslash{Cut(q)}$. Besides, for
$r<\mathrm{inj}(q)$, by (\ref{2.2}) we can obtain
\begin{eqnarray*}
\mathrm{vol}(B(q,r))=\int_{0}^{r}\int_{S_{q}^{n-1}}\det(\mathbb{A}(t,\xi))d\sigma{dt}.
\end{eqnarray*}
Denote by $\widehat{r}(x)=d_{M}(q,x)$ the intrinsic distance to the
point $q\in{M}$. Then, by the definition of a non-zero tangent
vector ``\emph{radial}" to a prescribed point on a manifold given in
the first page of \cite{KK}, we know that for
$x\in{M}\backslash(Cut(q)\cup{q})$ the unit vector field
\begin{eqnarray*}
v_{x}:=\nabla{\widehat{r}(x)}
\end{eqnarray*}
is the radial unit tangent vector at $x$. This is because for any
$\xi\in{S}_{q}^{n-1}$ and $t_{0}>0$, we have
$\nabla{\widehat{r}}{(\gamma_{\xi}(t_{0}))}=\gamma'_{\xi}(t_{0})$
when the point $\gamma_{\xi}(t_{0})=\exp_{q}(t_{0}\xi)$ is away from
the cut locus of $q$.

Set
\begin{eqnarray} \label{key-def1}
\ell(q):=\sup\limits_{x\in M}\widehat{r}(x)=\max\limits_{\xi\in
S_{q}M}d_{\xi}.
\end{eqnarray}
Clearly, $\ell(q)\geq\mathrm{inj}(q)$. By \cite[Proposition 39 on
page 266]{p}, one has
\begin{eqnarray*}
\partial_{\widehat{r}}\Delta{\widehat{r}}+\frac{(\Delta{\widehat{r}})^2}{n-1}\leq\partial_{\widehat{r}}\Delta{\widehat{r}}+|{\rm{Hess}}\widehat{r}|^{2}=-{\rm{Ric}}(\partial_{\widehat{r}},\partial_{\widehat{r}}),
\qquad {\rm{with}}~~\Delta{\widehat{r}}=\partial_{\widehat{r}}\ln(
\sqrt{\|g\|}),
\end{eqnarray*}
 with $\partial_{\widehat{r}}=\nabla{\widehat{r}}$ as a differentiable vector (cf.
\cite[Proposition 7 on page 47]{p} for the differentiation of
$\partial_{\widehat{r}}$), where $\rm{Hess}$, $\Delta$ are the
Hessian and the Laplace operators on $M$, respectively. Then,
together with (\ref{function-J}), we have
\begin{eqnarray} \label{FJ-1}
J''+\frac{1}{n-1}\mathrm{Ric}(\gamma_{\xi}'(t),\gamma_{\xi}'(t))J\leq0,
\end{eqnarray}
\begin{eqnarray} \label{FJ-2}
J(0,\xi)=0,\qquad J'(0,\xi)=1.
\end{eqnarray}

We need the following notion.

\begin{defn}  \label{def1-1}
Given a Riemannian manifold $M$ and a continuous function
$\lambda:[0,l)\rightarrow\mathbb{R}$, for a point $q\in M$, set
\begin{eqnarray}  \label{def1-1-F1}
\rho(q,x):=\left|\min\left\{0,\mathrm{Ric}(v_{x},v_{x})-(n-1)\lambda(\widehat{r}(x))\right\}\right|
\end{eqnarray}
 and
\begin{eqnarray} \label{def1-1-F2}
\widetilde{\rho}(q,x):=\left|\max\left\{0,\mathrm{Ric}(v_{x},v_{x})-(n-1)\lambda(\widehat{r}(x))\right\}\right|,
\end{eqnarray}
where $x\in M\backslash\left(Cut(q)\cup\{q\}\right)$, and, as
before, $\mathrm{Ric}$ denotes the Ricci curvature tensor on $M$,
$t=\widehat{r}(x)=d_{M}(q,x)$ denotes the Riemannian distance, on
$M$, from the point $q$ to $x$. We call two quantities
\begin{eqnarray*}
&&\quad k_{-}\left(p,q,\lambda(t),R\right)=\left(\int_{B(q,R)}\rho^{p}dv\right)^{1/p},\\
&&\overline{k_{-}}\left(p,q,\lambda(t),R\right)=\left(\frac{1}{\mathrm{vol}(B(q,R))}\int_{B(q,R)}\rho^{p}dv\right)^{1/p}
\end{eqnarray*}
as Type-I integral radial Ricci curvature (w.r.t. $q$) and its
average, respectively. Besides, we call two quantities
\begin{eqnarray*}
&&\quad
k^{\ast}_{-}\left(p,q,\lambda(t),R\right)=\left(\int_{B(q,R)}\widetilde{\rho}^{p}dv\right)^{1/p},
\\
&&\overline{k^{\ast}_{-}}\left(p,q,\lambda(t),R\right)=\left(\frac{1}{\mathrm{vol}(B(q,R))}\int_{B(q,R)}\widetilde{\rho}^{p}dv\right)^{1/p}
\end{eqnarray*}
as Type-II integral radial Ricci curvature (w.r.t. $q$) and its
average, respectively.
\end{defn}

Similarly, we have

\begin{defn}  \label{def1-2}
Given a Riemannian manifold $M$ and a continuous function
$\lambda:[0,l)\rightarrow\mathbb{R}$, for a point $q\in M$, set
\begin{eqnarray} \label{def1-2-F1}
\mu(q,x)=\left|\max\left\{0,K(V,v_{x})-\lambda(\widehat{r}(x))\right\}\right|
\end{eqnarray}
and
\begin{eqnarray} \label{def1-2-F2}
\widetilde{\mu}(q,x)=\left|\min\left\{0,K(V,v_{x})-\lambda(\widehat{r}(x))\right\}\right|,
\end{eqnarray}
where $x\in M\backslash\left(Cut(q)\cup\{q\}\right)$, $V\perp
v_{x}$, $V\in S_{x}^{n-1}\subseteq T_{x}M$, $K(V,v_{x})$ denotes the
sectional curvature of the plane spanned by $v_{x}$ and $V$, and, as
before, $t=\widehat{r}(x)=d_{M}(q,x)$ denotes the Riemannian
distance, on $M$, from the point $q$ to $x$. We call two quantities
\begin{eqnarray*}
&& \quad
k_{+}\left(p,q,\lambda(t),R\right)=\left(\int_{B(q,R)}\mu^{p}dv\right)^{1/p},
\\
&&\overline{k_{+}}\left(p,q,\lambda(t),R\right)=\left(\frac{1}{\mathrm{vol}(B(q,R))}\int_{B(q,R)}\mu^{p}dv\right)^{1/p}
\end{eqnarray*}
as Type-I integral radial sectional curvature (w.r.t. $q$) and its
average, respectively.  Besides, we call two quantities
\begin{eqnarray*}
&&\quad
k_{+}^{\ast}\left(p,q,\lambda(t),R\right)=\left(\int_{B(q,R)}\widetilde{\mu}^{p}dv\right)^{1/p},
\\
&&\overline{k_{+}^{\ast}}\left(p,q,\lambda(t),R\right)=\left(\frac{1}{\mathrm{vol}(B(q,R))}\int_{B(q,R)}\widetilde{\mu}^{p}dv\right)^{1/p}
\end{eqnarray*}
as Type-II integral radial sectional curvature (w.r.t. $q$) and its
average, respectively.
\end{defn}

\begin{remark} \label{remark2-1}
\rm{ (1) Clearly, if $k_{-}\left(p,q,\lambda(t),R\right)=0$ (resp.,
$k_{-}^{\ast}\left(p,q,\lambda(t),R\right)=0$), then
\begin{eqnarray}\label{rem1}
\mathrm{Ric}(v_{x},v_{x})\geq(n-1)\lambda(\widehat{r}(x)) \qquad
(\mathrm{resp.},~ ``\leq")
\end{eqnarray}
for $x\in B(q,R)\backslash\{q\}$ and $0<t<R$. This inequality can
also be rewritten as
$\mathrm{Ric}(\frac{d}{dt},\frac{d}{dt})\geq(n-1)\lambda(t)$ (resp.,
``$\leq$") for $0<t<R$, since
$\frac{d}{dt}|_{x}=\nabla\widehat{r}(x)=v_{x}$ is the radial unit
vector at $x$. We say that on $B(q,R)$, $M$ has a radial Ricci
curvature lower bound (resp., upper bound) $(n-1)\lambda(t)$ w.r.t.
$q$ if (\ref{rem1}) is satisfied. Similarly, if
$k_{+}\left(p,q,\lambda(t),R\right)=0$ (resp.,
$k_{+}^{\ast}\left(p,q,\lambda(t),R\right)=0$), then
 \begin{eqnarray} \label{rem2}
K(v_{x},V)=K(\frac{d}{dt},V)\leq\lambda(t) \qquad (\mathrm{resp.},~
``\geq")
 \end{eqnarray}
for $x\in B(q,R)\backslash\{q\}$ and $0<t<R$. We say that on
$B(q,R)$,  $M$ has a radial sectional curvature upper bound (resp.,
lower bound) $\lambda(t)$ w.r.t. $q$ provided (\ref{rem2}) is
satisfied. The notion that a manifold has radial (Ricci or
sectional) curvature bound w.r.t. a prescribed point has been
introduced in some literatures (see, e.g., \cite{fmi,m1,m2,m3}). For
a complete Riemannian manifold and a chosen point onside, one can
always find \emph{optimal} bounds for radial curvatures, which are
continuous functions of the Riemannian distance parameter. For this
fact, see \cite[expressions
(2.9), (2.10) on page 706]{fmi}.\\
 (2) Let $N\subset M$ be a $k$-dimensional submanifold of a
given Riemannian $n$-manifold. This submanifold does not have to be
closed or complete. One can define the \emph{normal tube of radius
$R$ around $N$} as follows
\begin{eqnarray} \label{rem-EXTRA}
\mathcal{T}(N,R):=\{x\in M|x=\exp_{N}(tv),~\mathrm{where}~0\leq
t<R,~v\in\upsilon(N),~\|v\|=1\},
\end{eqnarray}
where $\upsilon(N)$ is the normal bundle of $N$ in $M$ consisting of
vectors perpendicular to $N$, and $\exp_{N}:\upsilon(N)\rightarrow
M$ is the normal exponential map. For any point $q\in N$,
(\ref{def1-1-F1})-(\ref{def1-2-F2}) can be defined
similarly\footnote{ Of course, here we can define functions
$\mu(\widehat{r}(x))$, $\widetilde{\mu}(\widehat{r}(x))$,
$\rho(\widehat{r}(x))$ and $\widetilde{\rho}(\widehat{r}(x))$
similarly. However, in this setting, the point
$x\in\mathcal{T}(N,R)$ should be determined by $x=\exp_{N}(tv_0)$ or
$x=\exp_{N}(-t v_0)$ for some $v_{0}\in \upsilon(N)$. Speaking in
other words, for the case that $N\subset M$ is a geodesic,
$\widehat{r}(x)$ should be $\widehat{r}(x)=t=d_{M}(N,x)$, where
naturally $d_{M}(N,x)$ stands for the Riemannian distance $t$ from
the point $x$ to $N$.}, and then as Definition \ref{def1-1}, one can
define two quantities
\begin{eqnarray*}
&& \quad
k_{+}\left(p,N,\lambda(t),R\right)=\left(\int_{\mathcal{T}(N,R)}\mu^{p}dv\right)^{1/p},
\\
&&\overline{k_{+}}\left(p,N,\lambda(t),R\right)=\left(\frac{1}{\mathrm{vol}(\mathcal{T}(N,R))}\int_{\mathcal{T}(N,R)}\mu^{p}dv\right)^{1/p},
\end{eqnarray*}
which are called as \emph{Type-I integral radial sectional curvature
(w.r.t. $N$) and its average}, respectively. Other six quantities
$k_{+}^{\ast}\left(p,N,\lambda(t),R\right)$,
$\overline{k_{+}^{\ast}}\left(p,N,\lambda(t),R\right)$,
$k_{-}\left(p,N,\lambda(t),R\right)$,
$\overline{k_{-}}\left(p,N,\lambda(t),R\right)$,
$k_{-}^{\ast}\left(p,N,\lambda(t),R\right)$,
$\overline{k_{-}^{\ast}}\left(p,N,\lambda(t),R\right)$ can be
well-defined similarly. Clearly, when $N$ degenerates into a single
point, then these eight quantities here become exactly those ones
defined in
Definitions \ref{def1-1} and \ref{def1-2}. \\
 (3)  For a given
Riemannian manifold $M$, one can define two functions
$h_{1},h_{2}:M\rightarrow[0,\infty)$ as follows: $h_{1}(z)=$the
smallest eigenvalue for $\mathrm{Ric}: T_{z}M\rightarrow T_{z}M$,
$h_{2}(z)=$the smallest sectional curvature of a plane in $T_{z}M$.
Assume now that in Definitions \ref{def1-1} and \ref{def1-2},
$\lambda(t)$ degenerates into a constant $\lambda$, i.e.,
$\lambda(t)\equiv\lambda$. Then it is not hard to know
\begin{eqnarray} \label{rem3}
\quad k_{-}\left(p,q,\lambda,R\right)\leq\sup\limits_{x\in
M}\left(\int_{B(x,R)}|\min\{0,h_{1}(z)-(n-1)\lambda\}|^{p}dv\right)^{1/p},
\end{eqnarray}
\begin{eqnarray} \label{rem4}
\overline{k_{-}}\left(p,q,\lambda,R\right)\leq\sup\limits_{x\in
M}\left(\frac{1}{\mathrm{vol}(B(x,R))}\cdot\int_{B(x,R)}|\min\{0,h_{1}(z)-(n-1)\lambda\}|^{p}dv\right)^{1/p}
\end{eqnarray}
and
\begin{eqnarray} \label{rem5}
k_{+}^{\ast}\left(p,q,\lambda,R\right)\leq\sup\limits_{x\in
M}\left(\int_{B(x,R)}|\min\left\{0,h_{2}(z)-\lambda\right\}|^{p}dv\right)^{1/p},
\end{eqnarray}
\begin{eqnarray} \label{rem6}
\overline{k_{+}^{\ast}}\left(p,q,\lambda,R\right)\leq\sup\limits_{x\in
M}\left(\frac{1}{\mathrm{vol}(B(x,R))}\cdot\int_{B(x,R)}|\min\left\{0,h_{2}(z)-\lambda\right\}|^{p}dv\right)^{1/p}.
\end{eqnarray}
Similar conclusions can be obtained for other four quantities in
Definitions \ref{def1-1}, \ref{def1-2} and eight quantities defined
in (2) of Remark \ref{remark2-1}. \\
(4) In fact, the function $\lambda$ given in Definitions
\ref{def1-1} and \ref{def1-2} can be extended to a continuous
function on $M$. However, for accuracy, one needs to require $x\in
M\backslash\left(Cut(q)\cup\{q\}\right)$, since, as explained
before, only in this situation, there exists a unique minimizing
geodesic joining $q$ and $x$, which leads to the uniqueness of the
unit radial vector $v_{x}$ at $x$. If $x$ locates in the cut-locus
$Cut(q)$, one cannot construct $v_{x}$ at $x$ as before any more,
which leads to the consequence that the expressions
(\ref{def1-1-F1})-(\ref{def1-2-F2}) would not be correct any more.
Similar extension can be done for the case that $N\subset M$ is a
 submanifold of the given Riemannian manifold $M$.}
\end{remark}

We also need the following notion of spherically symmetric
manifolds.

\begin{defn} (see, e.g., \cite{fmi}) \label{defsp}
A domain
$\Omega=\exp_{q}([0,l)\times{S_{q}^{n-1}})\subset{M}\backslash
Cut(q)$, with $l<\mathrm{inj}(q)$, is said to be spherically
symmetric with respect to a point $q\in\Omega$, if the matrix
$\mathbb{A}(t,\xi)$ satisfies $\mathbb{A}(t,\xi)=f(t)I$, for a
function $f\in{C^{2}([0,l])}$, $l\in(0,\infty]$ with $f(0)=0$,
$f'(0)=1$, $f|_{(0,l)}>0$.
\end{defn}

By (\ref{2.1}), on the set
 $\Omega$ given in Definition \ref{defsp} the Riemannian metric of $M$ can be
 expressed by
 \begin{eqnarray} \label{2.3}
 ds^{2}(\exp_{q}(t\xi))=dt^{2}+f^{2}(t)\|d\xi\|^{2}, \qquad
 \xi\in{S_{q}^{n-1}}, \quad 0\leq{t}<l,
 \end{eqnarray}
 with $\|d\xi\|^{2}$ the round metric on the unit sphere $\mathbb{S}^{n-1}\subseteq\mathbb{R}^{n}$.
 Spherically symmetric manifolds were named as generalized space
forms
 by Katz and Kondo \cite{KK}, and a standard model for
such manifolds is given by the quotient manifold of the warped
product $[0,l)\times_{f} \mathbb{S}^{n-1}$ equipped with the metric
(\ref{2.3}),  and all pairs $(0,\xi)$ are identified with a single
point $q$, where $f$ satisfies the conditions in Definition
\ref{defsp}, and is called \emph{the warping function}. In this
setting, for $r<l$, one has
 \begin{eqnarray*}
 \mathrm{vol}(B(q,r))=w_{n}\int_{0}^{r}f^{n-1}(t)dt,
 \end{eqnarray*}
and, by the co-area formula, the area of the geodesic sphere
$\partial B(q,r)$ is
\begin{eqnarray*}
\mathrm{Area}(\partial
B(q,r))=\frac{d}{dr}\mathrm{vol}(B(q,r))=w_{n}f^{n-1}(r),
\end{eqnarray*}
where $w_n$ denotes the $(n-1)$-dimensional Hausdorff measure of
$\mathbb{S}^{n-1}\subset\mathbb{R}^n$. A space form with constant
curvature $k$ is also a spherically symmetric manifold, and in this
special case we have
\begin{eqnarray*}
f(t)=\left\{
\begin{array}{llll}
\frac{\sin\sqrt{k}t}{\sqrt{k}}, & \quad  l= \frac{\pi}{\sqrt{k}}
  & \quad k>0,\\
 t, &\quad l=+\infty & \quad k=0, \\
\frac{\sinh\sqrt{-k}t}{\sqrt{-k}}, & \quad l=+\infty  &\quad k<0.
\end{array}
\right.
\end{eqnarray*}

\section{Volume comparisons for manifolds with integral radial curvature
bounds} \label{VC}
\renewcommand{\thesection}{\arabic{section}}
\renewcommand{\theequation}{\thesection.\arabic{equation}}
\setcounter{equation}{0} \setcounter{maintheorem}{0}

Given a Riemannian manifold $M$ and a continuous function
$\lambda(t)$ of the distance parameter $t$ on $M$, we consider the
following system
\begin{eqnarray} \label{ODE}
\left\{
\begin{array}{lll}
f''(t)+\lambda(t)f(t)=0,  \qquad & 0<t<l,\\
f'(0)=1,~f(0)=0,\\
f(t)>0, \qquad & 0<t<l,
\end{array}
\right.
\end{eqnarray}
which will be used to determine our model spaces (i.e., spherically
symmetric manifolds).

We say that a function $\lambda$ on  $M$ satisfies a property
\textbf{P1} if

\begin{itemize}

\item  The function $\lambda(t)$ is continuous, non-positive on
$M\backslash\left(Cut(q)\cup\{q\}\right)$, where $q\in M$ is a given
point, and, as before, $t=\widehat{r}(x)=d_{M}(q,x)$ denotes the
Riemannian distance, on $M$, from the point $q$ to $x$. Besides, it
can be extended such that the extension is continuous (only w.r.t.
the distance parameter also), non-positive on $M$.

\end{itemize}

\begin{theorem} \label{theorem2-1}
Assume that $M$ is an $n$-dimensional ($n\geq2$) complete Riemannian
manifold, $q\in M$, and $\lambda(t)$ is a function on $M$ satisfying
the property \textbf{P1}. If $p>\frac{n}{2}$, then there exists a
positive constant $c(n,p,R)$, which is non-decreasing in $R$, such
that when $r<R$, we have
\begin{eqnarray*}
\left(\frac{\mathrm{vol}(B(q,R))}{\mathrm{vol}\left(\mathcal{B}_{n}(q^{-},R)\right)}\right)^{\frac{1}{2p}}-\left(\frac{\mathrm{vol}(B(q,r))}{\mathrm{vol}\left(\mathcal{B}_{n}(q^{-},r)\right)}\right)^{\frac{1}{2p}}\leq
c(n,p,R)\cdot\left(k_{-}\left(p,q,\lambda(t),R\right)\right)^\frac{1}{2},
\end{eqnarray*}
where $k_{-}\left(p,q,\lambda(t),R\right)$ is defined as in
Definition \ref{def1-1}, and $\mathcal{B}_{n}(q^{-},\cdot)$ denotes
the geodesic ball, with center $q^{-}$ and a prescribed radius, on
the spherically symmetric $n$-manifold
$M^{-}:=[0,\infty)\times_{f}\mathbb{S}^{n-1}$ with the base point
$q^{-}$ and the warping function $f$ determined by the system
(\ref{ODE}). Moreover, when $r=0$, we can obtain
 \begin{eqnarray*}
\mathrm{vol}(B(q,R))\leq\left(1+c(n,p,R)\cdot\left(k_{-}\left(p,q,\lambda(t),R\right)\right)^\frac{1}{2}\right)^{2p}\mathrm{vol}\left(\mathcal{B}_{n}(q^{-},R)\right).
 \end{eqnarray*}
\end{theorem}

\begin{remark} \label{remark3-2}
\rm{ (1) If $\lambda(t)\equiv\lambda$ is a non-positive constant,
then
\begin{eqnarray*}
\left\{
\begin{array}{ll}
f(t)=t, \qquad & \lambda=0,\\
f(t)=\sinh(\sqrt{-\lambda}t)/\sqrt{-\lambda}, \qquad & \lambda<0,
\end{array}
\right.
\end{eqnarray*}
 and this situation, the spherically symmetric $n$-manifold $M^{-}$
 degenerates into $\mathbb{R}^{n}$ or $\mathbb{H}^{n}(\lambda)$ (i.e., the hyperbolic $n$-space of constant sectional curvature
 $\lambda$), and, by (\ref{rem3}),
 Theorem \ref{theorem2-1}, we have
 \begin{eqnarray*}
&&\left(\frac{\mathrm{vol}(B(q,R))}{\mathrm{vol}\left(\mathcal{B}_{n}(R)\right)}\right)^{\frac{1}{2p}}-\left(\frac{\mathrm{vol}(B(q,r))}{\mathrm{vol}\left(\mathcal{B}_{n}(r)\right)}\right)^{\frac{1}{2p}}\leq
c(n,p,R)\cdot\left(k_{-}\left(p,q,\lambda,R\right)\right)^\frac{1}{2}\\
&&\qquad \qquad \leq c(n,p,R)\cdot\left[\sup\limits_{x\in
M}\left(\int_{B(x,R)}|\min\{0,h_{1}(z)-(n-1)\lambda\}|^{p}dv\right)^{\frac{1}{2p}}\right]\\
&&\qquad \qquad \leq
c(n,p,R)\cdot\left(\int_{M}|\min\{0,h_{1}(z)-(n-1)\lambda\}|^{p}dv\right)^{\frac{1}{2p}}
\end{eqnarray*}
 and
 \begin{eqnarray*}
\mathrm{vol}(B(q,R))&\leq&\left(1+c(n,p,R)\cdot\left(k_{-}\left(p,q,\lambda,R\right)\right)^\frac{1}{2}\right)^{2p}\mathrm{vol}\left(\mathcal{B}_{n}(R)\right)\\
&\leq&\left[1+c(n,p,R)\cdot\left(\int_{M}|\min\{0,h_{1}(z)-(n-1)\lambda\}|^{p}dv\right)^{\frac{1}{2p}}\right]^{2p}\mathrm{vol}\left(\mathcal{B}_{n}(R)\right),
 \end{eqnarray*}
 where $\mathcal{B}_{n}(R)$ is a geodesic ball, with radius $R$,
 in $\mathbb{R}^{n}$ or $\mathbb{H}^{n}(\lambda)$. Here it is not
 necessary to specify a center for the geodesic ball, since  $\mathbb{R}^{n}$ or
 $\mathbb{H}^{n}(\lambda)$ is two-point homogeneous. From the above
 argument, we know that even in the case $\lambda(t)\equiv\lambda$,
 conclusions of Theorem \ref{theorem2-1} are sharper than those in
 \cite[Theorem 1.1]{pw1}.
\\
(2) For the system (\ref{ODE}), since $\lambda(t)\leq0$, by the
Sturm-Picone separation theorem, we know that $l=\infty$, $f(t)\geq
t$ on $[0,\infty)$, and in this case, the model manifold $M^{-}$ has
the form $[0,\infty)\times_{f}\mathbb{S}^{n-1}$. Except the
non-positivity assumption of $\lambda(t)$, it is interesting to find
other conditions such that (\ref{ODE}) has a positive solution on
$(0,\infty)$. This problem has close relation with the oscillation
phenomenon of solutions of the ODE $f''(t)+\lambda(t)f(t)=0$. Mao
\cite[Subsection 2.6]{m1} investigated this problem and gave several
sufficient conditions such that (\ref{ODE}) has a positive solution
on $(0,\infty)$. For instance, he showed that if
$\lambda(t)\leq\frac{1}{4(t+1)^2}$ for $t>0$, then (\ref{ODE}) has a
positive solution on $(0,\infty)$. Consider an ODE
\begin{eqnarray*}
f''(t)+\frac{\gamma}{(t+1)^2}f(t)=0
\end{eqnarray*}
with $\gamma>0$ a constant, and it is not difficult to know that
this ODE is oscillatory for $\gamma>\frac{1}{4}$ but non-oscillatory
for $\gamma\leq\frac{1}{4}$. In this sense, the choice of the
function $\lambda(t)=\frac{1}{4(t+1)^2}$ somehow might be  critical
such that (\ref{ODE}) has a positive solution on $(0,\infty)$.
Several criterions for the non-existence of the long-time positive
solution to the system (\ref{ODE}) have also been given in
\cite[Subsection 2.6]{m1}. There is an interesting result we also
would like to mention here, that is, Hille \cite{eh} gave a nice
sufficient condition involving the $L^1$-norm of $\lambda(t)$ such
that (\ref{ODE}) has a positive solution on $(0,\infty)$.
\\
 (3) If $k_{-}\left(p,q,\lambda(t),R\right)\equiv0$, then
$\mathrm{Ric}(\frac{d}{dt},\frac{d}{dt})\geq(n-1)\lambda(t)$ for
$0<t<R$, where $\lambda(t)$ is a non-positive continuous function
w.r.t. the distance parameter $t$, and, by Theorem \ref{theorem2-1},
the following Bishop-Gromov type relative volume comparison estimate
\begin{eqnarray*}
\frac{\mathrm{vol}(B(q,R))}{\mathrm{vol}\left(\mathcal{B}_{n}(q^{-},R)\right)}\leq\frac{\mathrm{vol}(B(q,r))}{\mathrm{vol}\left(\mathcal{B}_{n}(q^{-},r)\right)},
\qquad r<R,
\end{eqnarray*}
can be obtained, which, by letting $r\rightarrow0$, yields
\begin{eqnarray*}
\mathrm{vol}(B(q,R))\leq\mathrm{vol}\left(\mathcal{B}_{n}(q^{-},R)\right),
\qquad R\geq0.
\end{eqnarray*}
The above two estimates have been shown in \cite{fmi,m1,m3} - for
details, see, e.g., \cite[Corollary 3.4]{m3}, and from which we know
that these two estimates are also valid without the non-positivity
assumption of $\lambda(t)$. Therefore, it is natural to ask a
problem as follows:

\begin{itemize}
\item (\textbf{Problem 1}) \emph{Similar conclusions to those of Theorem \ref{theorem2-1} might be obtained if the non-positivity of $\lambda(t)$ was
removed}.
\end{itemize}
(4) If one carefully check the proof of Theorem \ref{theorem2-1}
below, then it would be found that the non-positivity of
$\lambda(t)$ was used in (\ref{2-2-2}), that is, we need to use the
fact $f(t)\leq f(r)$ for $t\leq r$. However, in order to have
$f(t)\leq f(r)$ for $t\leq r$, it is not necessary to require the
non-positivity of $\lambda(t)$. In fact, since $f'(0)=1$, there
exists some $t_{0}>0$ such that $f'(t_0)=0$ and $f'(t)|_{[0,t_0)}>0$
that is, $t_0$ is the first positive zero point of $f'(t)$.
Therefore, one has $f(t)\leq f(r)$ for $0\leq t\leq r\leq t_0$,
which implies that one can also get the conclusion of Theorem
\ref{theorem2-1} if the non-positivity assumption of $\lambda(t)$
was removed, but only for the situation $R\leq t_0$. For instance,
if $\lambda(t)\equiv\lambda^{+}$ for some positive constant
$\lambda^{+}>0$, then
$f(t)=\frac{\sin(\sqrt{\lambda^{+}}t)}{\sqrt{\lambda^{+}}}$, and
similar conclusions to those of Theorem \ref{theorem2-1} can be
attained only for $r<R\leq\frac{\pi}{2\sqrt{\lambda^{+}}}$. Hence,
this fact gives an affirmative answer to \textbf{Problem 1} issued
above. \\
(5) In order to let the conclusion in Theorem \ref{theorem2-1} be
valid for any $R>0$, we require that $\lambda(t)$ satisfies the
property \textbf{P1}.
 }
\end{remark}

\begin{lemma} \label{lemma2-2}
Under assumptions of Theorem \ref{theorem2-1}, for the volume ratio
$\frac{\mathrm{vol}(B(q,r))}{\mathrm{vol}\left(\mathcal{B}_{n}(q^{-},r)\right)}=\frac{\mathrm{vol}(B(q,r))}{w_{n}\int_{0}^{r}f^{n-1}(t)dt}$,
we have
\begin{eqnarray*}
\frac{d}{dr}\frac{\mathrm{vol}(B(q,r))}{\mathrm{vol}\left(\mathcal{B}_{n}(q^{-},r)\right)}\leq
c_{1}(n,r)\left(\frac{\mathrm{vol}(B(q,r))}{\mathrm{vol}\left(\mathcal{B}_{n}(q^{-},r)\right)}\right)^{1-\frac{1}{2p}}\left(\int_{B(q,r)}\psi^{2p}dv\right)^{\frac{1}{2p}}
\left(\mathrm{vol}\left(\mathcal{B}_{n}(q^{-},r)\right)\right)^{-\frac{1}{2p}},
\end{eqnarray*}
where
\begin{eqnarray*}
&&c_{1}(n,r)=\max\limits_{t\in(0,r]}\frac{tf^{n-1}(t)}{\int_{0}^{t}f^{n-1}(s)ds},\\
&&c_{1}(n,0)=n,\\
&&\psi=\psi(t,\xi)=\max\left\{0,(n-1)\cdot\left[\frac{J'(t,\xi)}{J(t,\xi)}-\frac{f'(t)}{f(t)}\right]\right\}
\quad if ~0<t<d_{\xi}(q),
\end{eqnarray*}
 $\psi(t,\xi)=0$ if $t\geq d_{\xi}(q)$, and $J(t,\xi)$ is defined by (\ref{function-J}).
\end{lemma}

\begin{proof}
By direct calculation, on $\mathcal{D}_{q}$, which is diffeomorphic
to $M\backslash Cut(q)$, one has
\begin{eqnarray*}
\frac{d}{dt}\left(\frac{J^{n-1}(t,\xi)}{f^{n-1}(t)}\right)=(n-1)\frac{J^{n-1}(t,\xi)}{f^{n-1}(t)}\left(\frac{J'}{J}-\frac{f'}{f}\right)
\leq\psi(t,\xi)\frac{J^{n-1}(t,\xi)}{f^{n-1}(t)}.
\end{eqnarray*}
This inequality is also valid on the cut locus $Cut(q)$, since the
singular part of the derivative of $J^{n-1}(t,\xi)$ has negative
measure. Therefore, we have
 \begin{eqnarray*}
 \frac{d}{dr}\frac{\int_{\mathbb{S}^{n-1}}J^{n-1}(r,\xi)d\sigma}{\int_{\mathbb{S}^{n-1}}f^{n-1}(r)d\sigma}&=&\frac{\int_{\mathbb{S}^{n-1}}\frac{d}{dr}\left(\frac{J^{n-1}(r,\xi)}{f^{n-1}(r)}\right)d\sigma}{w_{n}}\\
 &\leq&\frac{1}{w_{n}}\int_{\mathbb{S}^{n-1}}\psi(r,\xi)\frac{J^{n-1}(r,\xi)}{f^{n-1}(r)}d\sigma,
 \end{eqnarray*}
 which implies that, for $t\leq r$, the following inequality
\begin{eqnarray} \label{2-2-1}
\frac{\int_{\mathbb{S}^{n-1}}J^{n-1}(r,\xi)d\sigma}{\int_{\mathbb{S}^{n-1}}f^{n-1}(r)d\sigma}-\frac{\int_{\mathbb{S}^{n-1}}J^{n-1}(t,\xi)d\sigma}{\int_{\mathbb{S}^{n-1}}f^{n-1}(t)d\sigma}\leq
\frac{1}{w_{n}}\int_{t}^{r}\int_{\mathbb{S}^{n-1}}\psi(r,\xi)\frac{J^{n-1}(r,\xi)}{f^{n-1}(r)}d\sigma
ds
\end{eqnarray}
holds. Since $\lambda(t)$ is a non-positive continuous function on
$(0,r)$, by the system (\ref{ODE}), it is easy to know that
$f''(t)\geq0$ on $(0,r)$, which implies $f'(t)\geq f'(0)=1$ for
$0<t<r$. Hence, one has $f(t)\leq f(r)$ provided $t\leq r$.
Combining this fact with (\ref{2-2-1}) yields
\begin{eqnarray} \label{2-2-2}
&&\int_{\mathbb{S}^{n-1}}J^{n-1}(r,\xi)d\sigma\cdot\int_{\mathbb{S}^{n-1}}f^{n-1}(t)d\sigma-\int_{\mathbb{S}^{n-1}}J^{n-1}(t,\xi)d\sigma\cdot\int_{\mathbb{S}^{n-1}}f^{n-1}(r)d\sigma\nonumber\\
&& \qquad
\leq \frac{1}{w_{n}}\int_{t}^{r}\int_{\mathbb{S}^{n-1}}\psi(r,\xi)\frac{J^{n-1}(r,\xi)}{f^{n-1}(r)}d\sigma ds\cdot\left(\int_{\mathbb{S}^{n-1}}f^{n-1}(r)d\sigma\right)\cdot\left(\int_{\mathbb{S}^{n-1}}f^{n-1}(t)d\sigma\right)\nonumber\\
&& \qquad \leq
f^{n-1}(r)\int_{t}^{r}\int_{\mathbb{S}^{n-1}}\psi(r,\xi)\frac{J^{n-1}(r,\xi)}{f^{n-1}(r)}d\sigma ds\cdot\left(\int_{\mathbb{S}^{n-1}}f^{n-1}(r)d\sigma\right)\nonumber\\
&& \qquad = \int_{t}^{r}\int_{\mathbb{S}^{n-1}}\psi(r,\xi)J^{n-1}(r,\xi)d\sigma ds\cdot\left(\int_{\mathbb{S}^{n-1}}f^{n-1}(r)d\sigma\right)\nonumber\\
&& \qquad \leq
w_{n}f^{n-1}(r)\int_{0}^{r}\int_{\mathbb{S}^{n-1}}\psi(r,\xi)J^{n-1}(r,\xi)d\sigma
ds \nonumber\\
 && \qquad = w_{n}f^{n-1}(r) \int_{B(q,r)}\psi dv \nonumber\\
&& \qquad \leq w_{n}f^{n-1}(r) \left(
\int_{B(q,r)}\psi^{2p}dv\right)^{\frac{1}{2p}}\cdot \left(
\int_{B(q,r)}dv\right)^{1-\frac{1}{2p}} \nonumber\\
&& \qquad = w_{n}f^{n-1}(r) \cdot
\left(\mathrm{vol}(B(q,r))\right)^{1-\frac{1}{2p}} \cdot\left(
\int_{B(q,r)}\psi^{2p}dv\right)^{\frac{1}{2p}},
\end{eqnarray}
where the last inequality holds by using Young's inequality
directly. Therefore, applying (\ref{2-2-2}), one can get
\begin{eqnarray}  \label{2-2-3}
\frac{d}{dr}\frac{\mathrm{vol}(B(q,r))}{\mathrm{vol}\left(\mathcal{B}_{n}(q^{-},r)\right)}&=&\frac{\left(\int_{\mathbb{S}^{n-1}}J^{n-1}(r,\xi)d\sigma\right)
\cdot\left(\int_{0}^{r}\int_{\mathbb{S}^{n-1}}f^{n-1}(t)d\sigma
dt\right)}
{\left(\mathrm{vol}\left(\mathcal{B}_{n}(q^{-},r)\right)\right)^2}\nonumber\\
&& \qquad -
\frac{\left(\int_{\mathbb{S}^{n-1}}f^{n-1}(r)d\sigma\right)
\cdot\left(\int_{0}^{r}\int_{\mathbb{S}^{n-1}}J^{n-1}(t,\xi)d\sigma
dt\right)}
{\left(\mathrm{vol}\left(\mathcal{B}_{n}(q^{-},r)\right)\right)^2}\nonumber\\
&=&
\frac{1}{\left(\mathrm{vol}\left(\mathcal{B}_{n}(q^{-},r)\right)\right)^2}\Bigg{[}\int_{0}^{r}\Bigg{(}\int_{\mathbb{S}^{n-1}}J^{n-1}(r,\xi)d\sigma\cdot\int_{\mathbb{S}^{n-1}}f^{n-1}(t)d\sigma\nonumber\\
&& \qquad -
\int_{\mathbb{S}^{n-1}}J^{n-1}(t,\xi)d\sigma\cdot\int_{\mathbb{S}^{n-1}}f^{n-1}(r)d\sigma\Bigg{)}dt\Bigg{]}\nonumber\\
&\leq& \frac{\int_{0}^{r}w_{n}f^{n-1}(r) \cdot
\left(\mathrm{vol}(B(q,r))\right)^{1-\frac{1}{2p}} \cdot\left(
\int_{B(q,r)}\psi^{2p}dv\right)^{\frac{1}{2p}}dt}{\left(\mathrm{vol}\left(\mathcal{B}_{n}(q^{-},r)\right)\right)^2}\nonumber\\
&=&\frac{w_{n}\cdot r \cdot f^{n-1}(r) \cdot
\left(\mathrm{vol}(B(q,r))\right)^{1-\frac{1}{2p}} \cdot\left(
\int_{B(q,r)}\psi^{2p}dv\right)^{\frac{1}{2p}}}{\left(\mathrm{vol}\left(\mathcal{B}_{n}(q^{-},r)\right)\right)^2}.
\end{eqnarray}
Since, by applying the L'H\^{o}pital's rule, we have
\begin{eqnarray*}
\lim\limits_{r\rightarrow0}\frac{w_{n}\cdot r \cdot
f^{n-1}(r)}{\mathrm{vol}\left(\mathcal{B}_{n}(q^{-},r)\right)}=\lim\limits_{r\rightarrow0}\frac{r
\cdot
f^{n-1}(r)}{\int_{0}^{r}f^{n-1}(t)dt}=\lim\limits_{r\rightarrow0}\frac{f^{n-1}(r)+r(n-1)f^{n-2}(r)f'(r)}{f^{n-1}(r)}=n,
\end{eqnarray*}
the maximum value $c_{1}(n,r)$ of $\frac{w_{n}\cdot r \cdot
f^{n-1}(r)}{\mathrm{vol}\left(\mathcal{B}_{n}(q^{-},r)\right)}$ can
be achieved on $(0,r]$. Especially, when $\lambda(x)\equiv0$, then
$f(t)=t$ for $0\leq t\leq r$, and in this case $c_{1}(n,r)=n$ for
all $r\geq0$. Putting this fact into (\ref{2-2-3}), one can easily
get the conclusion of Lemma \ref{lemma2-2}. The proof is finished.
\end{proof}

\begin{remark}
\rm{ Clearly, if $\lambda(t)\equiv\lambda$ is  a non-positive
constant, then the model space $M^{-}$ degenerates into
$\mathbb{R}^n$ or $\mathbb{H}^{n}(\lambda)$, and moreover, even in
this setting the conclusion of Lemma \ref{lemma2-2} also covers the
one of \cite[Lemma 2.1]{pw1} as a special case, since in
(\ref{def1-1-F1}) it only needs the radial Ricci curvature not the
Ricci curvature. Besides, we would like to point out one thing here,
that is, $c_{1}(n,r)$ has close relation with the function
$\lambda(t)$, since $c_{1}(n,r)$ also relies on the maximum value of
$f(t)$ on $(0,r]$, and meanwhile the warping function $f(t)$ is
determined by the system (\ref{ODE}). }
\end{remark}

We also need the following fact.

\begin{lemma} \label{lemma2-4}
There exists a constant $c_{2}(n,p)$, depending only on $n$ and $p$,
such that when $p>\frac{n}{2}$, we have
\begin{eqnarray*}
\int_{0}^{r}\psi^{2p}(t,\xi)\cdot J^{n-1}(t,\xi)dt\leq
c_{2}(n,p)\int_{0}^{r}\rho^{p}J^{n-1}(t,\xi)dt.
\end{eqnarray*}
\end{lemma}

\begin{proof}
Since $M$ is complete, it is easy to know that $\psi$ is absolutely
continuous. By direct calculation, one has
\begin{eqnarray*}
\psi'+\frac{\psi^{2}}{n-1}+\frac{2\psi
f'}{f}&=&(n-1)\max\Bigg{\{}0,\left[\frac{J''J-(J')^{2}}{J^2}-\frac{f''f-(f')^{2}}{f^2}\right]\\
&&\qquad
+\left[\frac{J'(t,\xi)}{J(t,\xi)}-\frac{f'(t)}{f(t)}\right]^{2}+2\left[\frac{J'(t,\xi)}{J(t,\xi)}-\frac{f'(t)}{f(t)}\right]\cdot\frac{f'}{f}\Bigg{\}}\\
&=&
\max\left\{0,(n-1)\left(\frac{J''}{J}-\frac{f''}{f}\right)\right\}\\
&\leq&\max\left\{0,(n-1)\lambda(t)-\mathrm{Ric}(v_{x},v_{x})\right\}\\
&=&\rho.
\end{eqnarray*}
Multiplying both sides of the above inequality by
$\psi^{2p-2}J^{n-1}$ and integrating over the interval $(0,r]$
yields
\begin{eqnarray} \label{2-4-1}
&&\int_{0}^{r}\psi'\psi^{2p-2}J^{n-1}dt+\frac{1}{n-1}\int_{0}^{r}\psi'\psi^{2p}J^{n-1}dt+2\int_{0}^{r}\frac{f'}{f}\psi^{2p-1}J^{n-1}dt \nonumber\\
&& \qquad\qquad\qquad\qquad\qquad \leq
\int_{0}^{r}\rho\cdot\psi^{2p-2}J^{n-1}dt.
\end{eqnarray}
On the other hand,
\begin{eqnarray*}
\int_{0}^{r}\psi'\psi^{2p-2}J^{n-1}dt&=&\frac{1}{2p-1}\psi^{2p-1}J^{n-1}\Bigg{|}_{0}^{r}-\frac{n-1}{2p-1}\int_{0}^{r}\psi^{2p-1}J^{n-2}J'dt\\
&\geq& -\frac{n-1}{2p-1}\int_{0}^{r}\psi^{2p-1}J^{n-2}J'dt\\
&\geq&
-\frac{1}{2p-1}\int_{0}^{r}\psi^{2p-1}J^{n-1}\left[\psi+(n-1)\frac{f'}{f}\right]dt\\
&=&-\frac{1}{2p-1}\int_{0}^{r}\psi^{2p}J^{n-1}dt-\frac{n-1}{2p-1}\int_{0}^{r}\psi^{2p-1}J^{n-1}\frac{f'}{f}dt.
\end{eqnarray*}
Substituting this inequality into (\ref{2-4-1}) results into
\begin{eqnarray*}
\left(\frac{1}{n-1}-\frac{1}{2p-1}\right)\int_{0}^{r}\psi^{2p}J^{n-1}dt+\left(2-\frac{n-1}{2p-1}\right)\int_{0}^{r}\psi^{2p-1}J^{n-1}\frac{f'}{f}dt\leq
\int_{0}^{r}\rho\cdot\psi^{2p-2}J^{n-1}dt.
\end{eqnarray*}
Therefore, when $p>\frac{n}{2}$, we have
$\frac{1}{n-1}-\frac{1}{2p-1}>0$, $0\leq2-\frac{n-1}{2p-1}<1$, and,
together with the Young's inequality, it follows that
\begin{eqnarray*}
\left(\frac{1}{n-1}-\frac{1}{2p-1}\right)\int_{0}^{r}\psi^{2p}J^{n-1}dt&\leq&\int_{0}^{r}\rho\cdot\psi^{2p-2}J^{n-1}dt\\
&\leq&\left(\int_{0}^{r}\rho^{p}J^{n-1}dt\right)^{\frac{1}{p}}\cdot\left(\int_{0}^{r}\psi^{2p}J^{n-1}dt\right)^{1-\frac{1}{p}},
\end{eqnarray*}
which implies
\begin{eqnarray*}
\left(\frac{1}{n-1}-\frac{1}{2p-1}\right)\left(\int_{0}^{r}\psi^{2p}J^{n-1}dt\right)^{\frac{1}{p}}\leq\left(\int_{0}^{r}\rho^{p}J^{n-1}dt\right)^{\frac{1}{p}},
\end{eqnarray*}
i.e.,
\begin{eqnarray*}
\int_{0}^{r}\psi^{2p}J^{n-1}dt\leq\left(\frac{1}{n-1}-\frac{1}{2p-1}\right)^{-p}\int_{0}^{r}\rho^{p}J^{n-1}dt.
\end{eqnarray*}
Hence, the assertion of Lemma \ref{lemma2-4} follows directly by
choosing $c_{2}:=\left(\frac{1}{n-1}-\frac{1}{2p-1}\right)^{-p}$.
\end{proof}

Now, we can give the proof of Theorem \ref{theorem2-1} as follows:

$\\$ \textbf{Proof of Theorem \ref{theorem2-1}}. Set
\begin{eqnarray*}
y(r):=\frac{\mathrm{vol}(B(q,r))}{\mathrm{vol}\left(\mathcal{B}_{n}(q^{-},r)\right)}.
\end{eqnarray*}
By Lemma \ref{lemma2-2}, one can easily get the following
differential inequality
\begin{eqnarray} \label{pt-1}
\frac{d}{dr}y(r)\leq
c_{1}(n,r)\left(\mathrm{vol}\left(\mathcal{B}_{n}(q^{-},r)\right)\right)^{-\frac{1}{2p}}\cdot
y^{1-\frac{1}{2p}}\cdot h(r)
\end{eqnarray}
where $h(r)$ is given by
\begin{eqnarray*}
h(r):=\left(\int_{B(q,r)}\psi^{2p}dv\right)^{\frac{1}{2p}}.
\end{eqnarray*}
It is clear that $y(0)>0$ for $r\geq0$, and
\begin{eqnarray*}
y(0)=\lim\limits_{r\rightarrow0}\frac{\mathrm{vol}(B(q,r))}{\mathrm{vol}\left(\mathcal{B}_{n}(q^{-},r)\right)}=1
\end{eqnarray*}
by applying  the L'H\^{o}pital's rule. Solving (\ref{pt-1}) directly
by separation of variables and integrating over the interval $[r,R]$
yield
\begin{eqnarray} \label{pt-2}
2p\cdot \left(y(R)\right)^{\frac{1}{2p}}-2p\cdot
\left(y(r)\right)^{\frac{1}{2p}}&\leq&\int_{r}^{R}c_{1}(n,s)\left(\mathrm{vol}\left(\mathcal{B}_{n}(q^{-},s)\right)\right)^{-\frac{1}{2p}}\cdot
h(s)ds \nonumber\\
&\leq& c_{1}(n,R)\cdot
c_{2}(n,p)\int_{r}^{R}\left(\mathrm{vol}\left(\mathcal{B}_{n}(q^{-},s)\right)\right)^{-\frac{1}{2p}}ds
\nonumber\\
&& \qquad \qquad \qquad
\cdot\left(k_{-}\left(p,q,\lambda(t),R\right)\right)^\frac{1}{2},
\end{eqnarray}
where the last inequality holds since $c_{1}(n,s)$ is non-decreasing
in $s$, and by Lemma \ref{lemma2-4}, one has
\begin{eqnarray*}
h(r)\leq
c_{2}(n,p)\left(\int_{B(q,R)}\rho^{p}dv\right)^{\frac{1}{2p}}=c_{2}(n,p)\cdot\left(k_{-}\left(p,q,\lambda(t),R\right)\right)^\frac{1}{2}.
\end{eqnarray*}
Set
\begin{eqnarray}  \label{EXP-C}
c(n,p,R):=\frac{1}{2p}c_{1}(n,R)\cdot
c_{2}(n,p)\int_{0}^{R}\left(\mathrm{vol}\left(\mathcal{B}_{n}(q^{-},s)\right)\right)^{-\frac{1}{2p}}ds,
\end{eqnarray}
which is well-defined\footnote{ Since $\lambda(t)\leq0$, by the
Sturm-Picone separation theorem, one has $f(t)\geq t$ on
$[0,\infty)$, which implies
\begin{eqnarray*}
0<\int_{r}^{R}\left(\mathrm{vol}\left(\mathcal{B}_{n}(q^{-},s)\right)\right)^{-\frac{1}{2p}}ds\leq\int_{r}^{R}\left(\frac{w_{n}s^{n}}{n}\right)^{-\frac{1}{2p}}ds<+\infty
\end{eqnarray*}
since $p>\frac{n}{2}$. Hence, from this fact, we know that
$\int_{r}^{R}\left(\mathrm{vol}\left(\mathcal{B}_{n}(q^{-},s)\right)\right)^{-\frac{1}{2p}}ds$
is convergent as $r\rightarrow0$. This means the constant $c(n,p,R)$
is well-defined.} and non-decreasing in $R$. Then from (\ref{pt-2}),
we can obtain
\begin{eqnarray*}
\left(\frac{\mathrm{vol}(B(q,R))}{\mathrm{vol}\left(\mathcal{B}_{n}(q^{-},R)\right)}\right)^{\frac{1}{2p}}-\left(\frac{\mathrm{vol}(B(q,r))}{\mathrm{vol}\left(\mathcal{B}_{n}(q^{-},r)\right)}\right)^{\frac{1}{2p}}\leq
c(n,p,R)\cdot\left(k_{-}\left(p,q,\lambda(t),R\right)\right)^\frac{1}{2},
\end{eqnarray*}
which is exactly the first assertion of Theorem \ref{theorem2-1}.
Furthermore, when $r=0$, $y(0)=1$, and then
\begin{eqnarray*}
\left(\frac{\mathrm{vol}(B(q,R))}{\mathrm{vol}\left(\mathcal{B}_{n}(q^{-},R)\right)}\right)^{\frac{1}{2p}}-1\leq
c(n,p,R)\cdot\left(k_{-}\left(p,q,\lambda(t),R\right)\right)^\frac{1}{2},
\end{eqnarray*}
which implies
\begin{eqnarray*}
\frac{\mathrm{vol}(B(q,R))}{\mathrm{vol}\left(\mathcal{B}_{n}(q^{-},R)\right)}\leq\left(1+c(n,p,R)\cdot\left(k_{-}\left(p,q,\lambda(t),R\right)\right)^\frac{1}{2}\right)^{2p}.
\end{eqnarray*}
The second assertion of Theorem \ref{theorem2-1} follows directly.
\hfill $\square$

Naturally, we can get the following volume doubling result.

\begin{corollary} \label{corollary3-6}
Given $D>0$, under assumptions of Theorem \ref{theorem2-1}, for all
$0<\alpha<1$, one can find
$$\epsilon:=\epsilon\left(n,p,q,\lambda(t),D,\alpha\right)>0$$
such that if $\ell(q)\leq D$, with $\ell(q)$ defined by
(\ref{key-def1}), and
$\overline{k_{-}}\left(p,q,\lambda(t),D\right)\leq\epsilon$, then
the inequality
\begin{eqnarray*}
\alpha\cdot\frac{\mathrm{vol}\left(\mathcal{B}_{n}(q^{-},r)\right)}{\mathrm{vol}\left(\mathcal{B}_{n}(q^{-},D)\right)}\leq\frac{\mathrm{vol}(B(q,r))}{\mathrm{vol}(M)}
\end{eqnarray*}
holds for $r<D$.
\end{corollary}

\begin{remark}
\rm{(1) The assumption $\ell(q)\leq D$ makes sure that $M\backslash
Cut(q)$ is bounded, which implies $\mathrm{vol}\left(M\backslash
Cut(q)\right)$ is finite. However, the complete manifold $M$ may be
unbounded. One can construct an interesting example as follows:

\begin{itemize}

\item Let $\mathfrak{C}$ be an $n$-dimensional cylinder embedded in
$\mathbb{R}^{n+1}$ and having finite height. For any point
$q\in\mathfrak{C}$, its cut-locus is a segment, consisting of its
antipodal points, which is actually a generatrix of $\mathfrak{C}$.
Extending this segment to be a line, and then the union of
$\mathfrak{C}$ and this line, denoting by $\mathcal{M}$, is a
complete manifold (still under the induced metric of $\mathfrak{C}$)
and is unbounded. Clearly, on $\mathcal{M}$, the cut-locus of $q$ is
just the line, and $\ell(q)<\infty$ is finite. Besides, it is easy
to see that $\mathrm{vol}(\mathcal{M})=\mathrm{vol}(\mathfrak{C})$.

\end{itemize}
Since $Cut(q)$ is a closed set of zero $n$-Hausdorff measure, one
has $\mathrm{vol}(M)=\mathrm{vol}\left(M\backslash Cut(q)\right)$,
which is finite, and this shows that the conclusion of Corollary
\ref{corollary3-6} makes sense. \\
(2) Since $\ell(q)\leq D$, $B(q,D)$ covers $M\backslash Cut(q)$,
which leads to the fact
\begin{eqnarray*}
\overline{k_{-}}\left(p,q,\lambda(t),D\right)=\left(\frac{1}{\mathrm{vol}(M)}\int_{M\backslash
Cut(q)}\rho^{p}dv\right)^{1/p}=\left(\frac{1}{\mathrm{vol}(M)}\int_{M}\rho^{p}dv\right)^{1/p}.
\end{eqnarray*}
 }
\end{remark}

\begin{proof}
We divide the proof into two cases as follows:

Case 1. If $\ell(q)\leq r<D$, then $B(q,r)$ covers $M\backslash
Cut(q)$, which implies $\mathrm{vol}(B(q,r))=\mathrm{vol}(M)$.
Hence, for $0<\alpha<1$, the inequality
\begin{eqnarray*}
\alpha\cdot\frac{\mathrm{vol}\left(\mathcal{B}_{n}(q^{-},r)\right)}{\mathrm{vol}\left(\mathcal{B}_{n}(q^{-},D)\right)}\leq\frac{\mathrm{vol}(B(q,r))}{\mathrm{vol}(M)}=1
\end{eqnarray*}
holds \emph{trivially}.

Case 2. If $r<\ell(q)\leq D$, then $B(q,D)$ covers $M\backslash
Cut(q)$, and $\mathrm{vol}(B(q,D))=\mathrm{vol}(M)$. By Theorem
\ref{theorem2-1}, one can get
\begin{eqnarray*}
\left(\frac{\mathrm{vol}(B(q,D))}{\mathrm{vol}\left(\mathcal{B}_{n}(q^{-},D)\right)}\right)^{\frac{1}{2p}}-\left(\frac{\mathrm{vol}(B(q,r))}{\mathrm{vol}\left(\mathcal{B}_{n}(q^{-},r)\right)}\right)^{\frac{1}{2p}}&\leq&
c(n,p,D)\cdot\left(k_{-}\left(p,q,\lambda(t),D\right)\right)^\frac{1}{2}\\
&=& c(n,p,D)\cdot\left(\int_{M\backslash
Cut(q)}\rho^{p}dv\right)^\frac{1}{2p},
\end{eqnarray*}
which implies
\begin{eqnarray} \label{3-6-1}
&&\left(\frac{\mathrm{vol}\left(\mathcal{B}_{n}(q^{-},r)\right)}{\mathrm{vol}\left(\mathcal{B}_{n}(q^{-},D)\right)}\right)^{\frac{1}{2p}}-\left(\frac{\mathrm{vol}(B(q,r))}{\mathrm{vol}(M)}\right)^{\frac{1}{2p}}\nonumber\\
&&\qquad \qquad \leq
c(n,p,D)\cdot\left(\mathrm{vol}\left(\mathcal{B}_{n}(q^{-},r)\right)\right)^{\frac{1}{2p}}\cdot\left(\frac{1}{\mathrm{vol}(M)}\int_{M\backslash
Cut(q)}\rho^{p}dv\right)^{\frac{1}{2p}}.
\end{eqnarray}
Choose $\epsilon$ to be a positive constant such that
\begin{eqnarray*}
\left(c(n,p,D)\right)^{2p}\cdot\epsilon^{p}\leq\frac{(1-\alpha)^{2p}}{\mathrm{vol}\left(\mathcal{B}_{n}(q^{-},D)\right)}.
\end{eqnarray*}
If
\begin{eqnarray*}
\overline{k_{-}}\left(p,q,\lambda(x),D\right)=\left(\frac{1}{\mathrm{vol}(M)}\int_{M\backslash
Cut(q)}\rho^{p}dv\right)^{1/p}\leq\epsilon,
\end{eqnarray*}
then, by (\ref{3-6-1}), one has
\begin{eqnarray*}
&&\left(\frac{\mathrm{vol}\left(\mathcal{B}_{n}(q^{-},r)\right)}{\mathrm{vol}\left(\mathcal{B}_{n}(q^{-},D)\right)}\right)^{\frac{1}{2p}}-\left(\frac{\mathrm{vol}(B(q,r))}{\mathrm{vol}(M)}\right)^{\frac{1}{2p}}\nonumber\\
&&\qquad \qquad \leq
\epsilon^{-\frac{1}{2}}\cdot\frac{1-\alpha}{\left(\mathrm{vol}\left(\mathcal{B}_{n}(q^{-},D)\right)\right)^{\frac{1}{2p}}}\cdot\left(\mathrm{vol}\left(\mathcal{B}_{n}(q^{-},r)\right)\right)^{\frac{1}{2p}}\cdot\left(\frac{1}{\mathrm{vol}(M)}\int_{M\backslash
Cut(q)}\rho^{p}dv\right)^{\frac{1}{2p}}\\
&&\qquad \qquad \leq
(1-\alpha)\cdot\left(\frac{\mathrm{vol}\left(\mathcal{B}_{n}(q^{-},r)\right)}{\mathrm{vol}\left(\mathcal{B}_{n}(q^{-},D)\right)}\right)^{\frac{1}{2p}},
\end{eqnarray*}
which implies
\begin{eqnarray*}
\alpha\cdot\frac{\mathrm{vol}\left(\mathcal{B}_{n}(q^{-},r)\right)}{\mathrm{vol}\left(\mathcal{B}_{n}(q^{-},D)\right)}\leq\frac{\mathrm{vol}(B(q,r))}{\mathrm{vol}(M)}.
\end{eqnarray*}

The conclusion of Corollary \ref{corollary3-6} follows by summing up
Cases 1 and 2 directly.
\end{proof}

We can also get a local version of Theorem \ref{theorem2-1}, which
can be seen as an extension of \cite[Theorem 2.1]{pw2},  as follows:

\begin{corollary} \label{corollary3-8}
Given $0<\alpha<1$, under assumptions of Theorem \ref{theorem2-1},
there exists an $\epsilon=\epsilon(n,p,\alpha)$ such that if
$R^{2}\cdot\overline{k_{-}}\left(p,q,0,R\right)<\epsilon$, then for
$r_{1}<r_{2}\leq R$, the following relative volume comparison
\begin{eqnarray*}
\alpha\cdot\left(\frac{r_{1}}{r_{2}}\right)^{n}\leq\frac{\mathrm{vol}(B(q,r_{1}))}{\mathrm{vol}(B(q,r_{2}))}
\end{eqnarray*}
holds.
\end{corollary}

\begin{proof}
Since $\lambda(t)\equiv0$, $f(t)=t$ for $t>0$, and the spherically
symmetric $n$-manifold degenerates into $\mathbb{R}^n$. In this
case,
$\mathrm{vol}\left(\mathcal{B}_{n}(q^{-},r)\right)=\mathrm{vol}\left(\mathcal{B}_{n}(r)\right)=\frac{w_{n}r^{n}}{n}$.
As shown in the proof of Theorem \ref{theorem2-1}, by using Lemmas
\ref{lemma2-2}, \ref{lemma2-4}, and choosing $\lambda(t)\equiv0$,
one can obtain
\begin{eqnarray*}
\left(y(r_2)\right)^{\frac{1}{2p}}-
\left(y(r_1)\right)^{\frac{1}{2p}}&\leq& n\cdot
c_{2}(n,p)\int_{r_1}^{r_2}\left(\mathrm{vol}\left(\mathcal{B}_{n}(s)\right)\right)^{-\frac{1}{2p}}ds
\cdot\left(k_{-}\left(p,q,0,R\right)\right)^\frac{1}{2}\\
&\leq&n\cdot
c_{2}(n,p)\int_{0}^{R}\left(\mathrm{vol}\left(\mathcal{B}_{n}(s)\right)\right)^{-\frac{1}{2p}}ds
\cdot\left(k_{-}\left(p,q,0,R\right)\right)^\frac{1}{2}\\
&=&\frac{n^{1+\frac{1}{2p}}}{2p-n}\cdot
w_{n}^{-\frac{1}{2p}}c_{2}(n,p)\cdot\left(k_{-}\left(p,q,0,R\right)\right)^\frac{1}{2}R^{1-\frac{n}{2p}},
\end{eqnarray*}
where, similar as before,
$y(r)=\mathrm{vol}(B(q,r))/\mathrm{vol}\left(\mathcal{B}_{n}(r)\right)$.
Furthermore, one has
\begin{eqnarray} \label{3-8-1}
\left(\frac{\mathrm{vol}(B(q,r_2))}{r_{2}^{n}}\right)^{\frac{1}{2p}}-\left(\frac{\mathrm{vol}(B(q,r_1))}{r_{1}^{n}}\right)^{\frac{1}{2p}}
\leq\frac{n}{2p-n}\cdot
c_{2}(n,p)\cdot\left(k_{-}\left(p,q,0,R\right)\right)^\frac{1}{2}R^{1-\frac{n}{2p}},
\end{eqnarray}
which implies
\begin{eqnarray} \label{3-8-2}
\left(\frac{r_{1}}{r_{2}}\right)^{\frac{n}{2p}}-\left(\frac{\mathrm{vol}(B(q,r_{1}))}{\mathrm{vol}(B(q,r_{2}))}\right)^{\frac{1}{2p}}&\leq&
\left(\frac{r_{1}^n}{\mathrm{vol}(B(q,r_{1}))}\right)^{\frac{1}{2p}}\frac{n}{2p-n}\cdot
c_{2}(n,p)\cdot\left(k_{-}\left(p,q,0,R\right)\right)^\frac{1}{2}R^{1-\frac{n}{2p}}\nonumber\\
&=&c_{3}(n,p,r_{2},R)\cdot\left(\frac{r_{1}}{r_{2}}\right)^{\frac{n}{2p}},
\end{eqnarray}
where
\begin{eqnarray*}
c_{3}(n,p,r_{2},R):=\left(\frac{r_{2}^n}{\mathrm{vol}(B(q,r_{1}))}\right)^{\frac{1}{2p}}\frac{n}{2p-n}\cdot
c_{2}(n,p)\cdot\left(k_{-}\left(p,q,0,R\right)\right)^\frac{1}{2}R^{1-\frac{n}{2p}}.
\end{eqnarray*}
Clearly, $c_{3}(n,p,r_{2},R)$ can be seen as a function of $r_2$,
which is non-decreasing. From (\ref{3-8-2}), it follows that
\begin{eqnarray} \label{3-8-3}
\left(1-c_{3}(n,p,r_{2},R)\right)\left(\frac{r_{1}}{r_{2}}\right)^{\frac{n}{2p}}\leq\left(\frac{\mathrm{vol}(B(q,r_{1}))}{\mathrm{vol}(B(q,r_{2}))}\right)^{\frac{1}{2p}}.
\end{eqnarray}
Replacing $r_1$, $r_2$ in (\ref{3-9-1}) by $r_2$, $R$, we have
\begin{eqnarray*}
\left(\frac{r_{2}^{n}}{\mathrm{vol}(B(q,r_2))}\right)^{\frac{1}{2p}}&\leq&\left(\left(\frac{\mathrm{vol}(B(q,R))}{R^{n}}\right)^{\frac{1}{2p}}-\frac{n}{2p-n}\cdot
c_{2}(n,p)\cdot\left(k_{-}\left(p,q,0,R\right)\right)^\frac{1}{2}R^{1-\frac{n}{2p}}\right)^{-1}\\
&=&\left(\frac{R^{n}}{\mathrm{vol}(B(q,R))}\right)^{\frac{1}{2p}}\left(1-\frac{n}{2p-n}\cdot
c_{2}(n,p)\cdot\left(\overline{k_{-}}\left(p,q,0,R\right)\right)^\frac{1}{2}R\right)^{-1}\\
&\leq&\left(\frac{R^{n}}{\mathrm{vol}(B(q,R))}\right)^{\frac{1}{2p}}\left(1-\frac{n}{2p-n}\cdot
c_{2}(n,p)\epsilon^\frac{1}{2}\right)^{-1}\\
&\leq&2\left(\frac{R^{n}}{\mathrm{vol}(B(q,R))}\right)^{\frac{1}{2p}}
\end{eqnarray*}
if $\epsilon\leq\left(\frac{2p-n}{2n\cdot
c_{2}(n,p)}\right)^{2}\cdot\left(1-\alpha^{\frac{1}{2p}}\right)^2$.
Hence, one has
 \begin{eqnarray*}
1-c_{3}(n,p,r_{2},R)&\geq&1-2\left(\frac{R^{n}}{\mathrm{vol}(B(q,R))}\right)^{\frac{1}{2p}}\frac{n}{2p-n}\cdot
c_{2}(n,p)\cdot\left(k_{-}\left(p,q,0,R\right)\right)^\frac{1}{2}R^{1-\frac{n}{2p}}\\
&=&1-\frac{2n}{2p-n}\cdot
c_{2}(n,p)\cdot\left(\overline{k_{-}}\left(p,q,0,R\right)\right)^\frac{1}{2}R\\
&\geq&1-\frac{2n}{2p-n}\cdot c_{2}(n,p)\cdot\epsilon^{\frac{1}{2}}\\
&\geq&1-\left(1-\alpha^{\frac{1}{2p}}\right)\\
&=&\alpha^{\frac{1}{2p}}.
 \end{eqnarray*}
Substituting the above inequality into (\ref{3-8-2}) yields
\begin{eqnarray*}
\alpha^{\frac{1}{2p}}\left(\frac{r_{1}}{r_{2}}\right)^{\frac{n}{2p}}\leq(1-c_{3}(n,p,r_{2},R))\left(\frac{r_{1}}{r_{2}}\right)^{\frac{n}{2p}}
\leq\left(\frac{\mathrm{vol}(B(q,r_{1}))}{\mathrm{vol}(B(q,r_{2}))}\right)^{\frac{1}{2p}},
\end{eqnarray*}
which implies the conclusion of Corollary \ref{corollary3-8}
directly.
\end{proof}

Now, we would like to give an upper bound for the volume of normal
tubes around geodesics.

\begin{theorem} \label{theorem3-8}
If $N\subset M$ is a geodesic of the given complete Riemannian
$n$-manifold $M$, $n>2$, $p>n-1$, then the volume of the normal tube
of radius $R$ around  $N$ can be estimated as follows
\begin{eqnarray*}
\mathrm{vol}\left(\mathcal{T}(N,R)\right)\leq(a+b)^{n-2}\int_{0}^{R}\left(\frac{e^{bt}-1}{b}\right)^{n-2}dt,
\end{eqnarray*}
where $\mathcal{T}(N,R)$ is defined as (\ref{rem-EXTRA}),
\begin{eqnarray*}
&&a=\left(\mathbb{L}_{N}\cdot\mathrm{vol}(\mathbb{S}^{n-2})\right)^{\frac{1}{n-2}},\\
&& b= \left(c_{4}(n,p)\right)^{\frac{1}{2p-1}}\cdot\left(
k^{\ast}_{+}(p,N,0,R)\right)^{\frac{p}{2p-1}},
\end{eqnarray*}
with
\begin{eqnarray*}
&&\mathbb{L}_{N}~\mathrm{the~length~of~the~geodesic}~N, \\
&&c_{4}(n,p)=\frac{(2p-1)^{p}(p-1)^{p}}{(n-2)^{p}p^{p}(2p-n+1)^{p-1}}\left[2\left(\frac{2p-1}{p+1-n}\right)^{p}+1\right],\\
&&
k^{\ast}_{+}(p,N,0,R)=\left(\int_{\mathcal{T}(N,R)}\left|\min\left\{0,K(V,v_{x})\right\}\right|^{p}dv\right)^{1/p}.
\end{eqnarray*}
\end{theorem}

\begin{proof}
Define a quantity $\ell(N)$ as follows
\begin{eqnarray*}
\ell(N):=\sup\limits_{q\in N}\ell(q).
\end{eqnarray*}
As mentioned in (2) of Remark \ref{remark2-1}, one can define
Type-II integral radial sectional curvature w.r.t. $N$ as follows
\begin{eqnarray} \label{DEF-ISK}
k_{+}^{\ast}\left(p,N,\lambda(t),R\right)=\left(\int_{\mathcal{T}(N,R)}\left|\min\left\{0,K(V,v_{x})-\lambda(t)\right\}\right|^{p}dv\right)^{1/p},
\qquad \forall R>0.
\end{eqnarray}
Clearly, if $R\geq\ell(N)$, then $\mathcal{T}(N,R)$ covers $M$
completely, and in this situation, we have
\begin{eqnarray*}
k_{+}^{\ast}\left(p,N,\lambda(t),R\right)=\left(\int_{M}\left|\min\left\{0,K(V,v_{x})-\lambda(t)\right\}\right|^{p}dv\right)^{1/p}:=k^{\ast}_{+}(p,N,\lambda(t)).
\end{eqnarray*}
Hence, one has
\begin{eqnarray*}
k_{+}^{\ast}\left(p,N,\lambda(t),R\right)\leq
k^{\ast}_{+}(p,N,\lambda(t)), \qquad \forall R>0.
\end{eqnarray*}
Since $N$ is a geodesic, by a parallel moving along $N$,  one can
set up a coordinates system $(t,s,\theta_{n-2})$  for
$\mathcal{T}(N,R)$, where $t=d_{M}(N,\cdot)$ denotes the radial
distance to $N$, $s$ is the arc-length parameter of $N$, and
$\theta_{n-2}\in\mathbb{S}^{n-2}$ is the angular parameter from the
unit normal bundle. Write the Riemannian volume element as
$dv=wdtdsd\theta_{n-2}$. Define
$(\mathcal{V}(r))^{n-2}=\int_{\mathbb{S}^{n-2}}\int_{N}wdsd\theta_{n-2}=$
area of level set $t=r$. One checks \cite[pages 1015-1019]{pw11}
carefully and can easily find that only radial curvatures have been
used in the argument therein. Therefore, by using a similar argument
as that in \cite[pages 1015-1019]{pw11}, one has
 \begin{eqnarray} \label{SOLD}
&&\mathcal{V}'\leq a+b\mathcal{V}^{\delta},\\
&&\quad \mathcal{V}(0)=0, \nonumber\\
&&\delta=\frac{2p-n+1}{2p-1}\in(0,1), \nonumber
 \end{eqnarray}
which implies
\begin{eqnarray*}
\mathcal{V}'\leq(a+b)+b\mathcal{V}, \qquad \mathcal{V}(0)=0.
\end{eqnarray*}
Solving the above differential inequality yields
\begin{eqnarray*}
\mathcal{V}(t)\leq(a+b)\left(\frac{e^{bt}-1}{b}\right).
\end{eqnarray*}
Together with the fact
\begin{eqnarray*}
\mathrm{vol}\left(\mathcal{T}(N,R)\right)=\int_{0}^{R}\mathcal{V}^{n-2}(t)dt,
\end{eqnarray*}
one can get the conclusion of Theorem \ref{theorem3-8} directly.
\end{proof}

Applying Theorem \ref{theorem3-8}, we can give an upper bound for
$\mathrm{vol}\left(\mathcal{T}(N,R)\right)$ in terms of
$k^{\ast}_{+}(p,N,\lambda(t))$, where
$\lambda(\widehat{r}(x))=\lambda(t)$ is a continuous function
defined on $M\setminus\left(Cut(N)\cup N\right)$, with
$Cut(N):=\bigcup\limits_{q\in N}Cut(q)$ and, as before,
$t=d_{M}(N,x)$ the Riemannian distance from $x$ to $N$. In order to
state our conclusion clearly, we need a concept. We say that a
function $\lambda$ on a given Riemannian manifold $M$ satisfies a
property \textbf{P2} if
\begin{itemize}

\item The function $\lambda(t)$ is continuous, non-positive on
$M\backslash\left(Cut(N)\cup N\right)$, where $N\subset M$ is a
given submanifold of $M$, and, as mentioned in (2) of Remark
\ref{remark2-1}, $t=\widehat{r}(x)=d_{M}(N,x)$ denotes the
Riemannian distance, on $M$, from the point $x$ to $N$. Besides, it
can be extended such that the extension is continuous (only w.r.t.
the distance parameter also), non-positive on $M$.

\end{itemize}

\begin{theorem} \label{theorem3-9}
Under the assumptions of Theorem \ref{theorem3-8} and the assumption
that $\lambda(t)$ is a function on $M$ satisfying the property
\textbf{P2}, we have
\begin{eqnarray*}
\mathrm{vol}\left(\mathcal{T}(N,R)\right)\leq\digamma(n,p,\mathbb{L}_{N},\lambda_{\inf},k_{+}^{\ast}(N),R),
\end{eqnarray*}
where $k_{+}^{\ast}(N):=k_{+}^{\ast}\left(p,N,\lambda(t),R\right)$,
given by (\ref{DEF-ISK}), is the Type-II integral radial sectional
curvature w.r.t. $N$, $\lambda_{\inf}:=\inf_{(0,R)}\lambda(t)\leq0$,
and, as before, $\mathbb{L}_{N}$ is the length of $N$. Moreover, as
$\mathbb{L}_{N}, k_{+}^{\ast}(N)\rightarrow0$, we have
$\digamma(n,p,\mathbb{L}_{N},\lambda_{\inf},k_{+}^{\ast}(N),R)\rightarrow0$.
\end{theorem}

\begin{proof}
As in \cite[page 1018]{pw11}, one can define a function
$\sigma:=\max\{0,\min(\mathrm{eign}(-R(\cdot,\frac{d}{dt})\frac{d}{dt}))\}$,
with $\mathrm{eign}(-R(\cdot,\frac{d}{dt})\frac{d}{dt})$ the set of
eigenvalues of the operator $V\mapsto
-R(V,\frac{d}{dt})\frac{d}{dt}$, where for any $x\in
M\setminus\left(Cut(N)\cup N\right)$, $\frac{d}{dt}|_{x}=v_{x}$,
$V\perp v_{x}$, $V\in S_{x}^{n-1}\subseteq T_{x}M$. Clearly, for any
point $x$, $\min(\mathrm{eign}(-R(\cdot,\frac{d}{dt})\frac{d}{dt}))$
is actually the \emph{minimal} eigenvalue of  $-K(V,v_{x})$, with,
as before, $K(V,v_{x})$ the radial sectional curvature at $x$. Then
we have $\sigma\leq\max\{\sigma+\lambda(t),0\}-\lambda(t)$, which
implies
\begin{eqnarray} \label{3-9-1}
\sigma^{p}\leq2^{p-1}\left((\max\{\sigma+\lambda(t),0\})^{p}+|\lambda(t)|^{p}\right).
\end{eqnarray}
Combining (\ref{3-9-1}) with (\ref{SOLD}), we have
\begin{eqnarray} \label{3-9-2}
&&\mathcal{V}'(R)\leq
\left(\mathbb{L}_{N}\cdot\mathrm{vol}(\mathbb{S}^{n-2})\right)^{\frac{1}{n-2}}
\nonumber\\
&&\qquad +\left(2^{p-1}c_{4}(n,p)\left(\int_{\mathbb{S}^{n-2}}\int_{N}\int_{0}^{R}|\lambda(t)|^{p}wdtdsd\theta_{n-2}+\left(k_{+}^{\ast}\left(p,N,\lambda(t),R\right)\right)^{p}\right)\right)^{\frac{1}{2p-1}}\mathcal{V}^{\delta}\nonumber\\
&&~~=\left(\mathbb{L}_{N}\cdot\mathrm{vol}(\mathbb{S}^{n-2})\right)^{\frac{1}{n-2}}
\nonumber\\
&&\qquad+\left(c_{5}(n,p)\left(|\lambda_{\inf}|^{p}\int_{0}^{R}\mathcal{V}^{n-2}(t)dt+\left(k_{+}^{\ast}\left(p,N,\lambda(t),R\right)\right)^{p}\right)\right)^{\frac{1}{2p-1}}\mathcal{V}^{\delta},
\end{eqnarray}
where $c_{5}(n,p):=2^{p-1}c_{4}(n,p)$,
 and other notations have the same meanings as those
in Theorem \ref{theorem3-8}. From (\ref{3-9-2}) and the definition
of $\mathcal{V}$, one has
\begin{eqnarray}
\mathcal{V}'(R)\leq
a+\left(\widetilde{b}\int_{0}^{R}\mathcal{V}^{n-2}(t)dt+\widetilde{c}\right)^{\frac{1}{2p-1}}\mathcal{V}^{\delta},
\qquad \mathcal{V}(0)=0,
\end{eqnarray}
where $a$ is defined as in Theorem \ref{theorem3-8},
$\widetilde{b}=c_{5}(n,p)|\lambda_{\inf}|^{p}$,
$\widetilde{c}=c_{5}(n,p)\cdot\left(k_{+}^{\ast}\left(p,N,\lambda(t),R\right)\right)^{p}$.
Using a similar argument as that in \cite[page 1021]{pw11}, we know
that
$\mathrm{vol}\left(\mathcal{T}(N,R)\right)=\int_{0}^{R}\mathcal{V}^{n-2}(t)dt$
satisfies
\begin{eqnarray} \label{3-9-3}
\frac{d}{dR}\mathrm{vol}\left(\mathcal{T}(N,R)\right)\leq\alpha+\beta+\gamma\cdot\mathrm{vol}\left(\mathcal{T}(N,R)\right),
\end{eqnarray}
where
\begin{eqnarray*}
&&
\alpha=a^{\frac{n-2}{n-1-\delta}}+\left(\frac{2p-1}{2p\widetilde{b}}\right)^{\frac{n-2}{n-1-\delta}}\widetilde{c},\\
&& \beta=a^{\frac{n-2}{n-1-\delta}},\\
&&\gamma=
a^{\frac{n-2}{n-1-\delta}}+\left(\frac{2p-1}{2p}\right)^{\frac{n-2}{n-1-\delta}}\widetilde{b}^{\frac{1-\delta}{n-1-\delta}},\\
 &&\delta=\frac{2p-n+1}{2p-1}.
\end{eqnarray*}
Solving the differential inequality (\ref{3-9-3}), and together with
the initial condition $\mathrm{vol}\left(\mathcal{T}(N,0)\right)=0$,
directly yields
\begin{eqnarray*}
\mathrm{vol}\left(\mathcal{T}(N,R)\right)\leq\frac{\alpha\gamma+\beta}{\gamma^2}\left(e^{\gamma
R}-1\right)-\frac{\beta}{\gamma}R,
\end{eqnarray*}
which implies the conclusion of Theorem \ref{theorem3-9}.
\end{proof}

For a complete Riemannian $n$-manifold and a subset
$\widehat{S}_{q}\subset S_{q}M\subset T_{q}M$ of $T_{q}M$, set
\begin{eqnarray*}
\Gamma(\widehat{S}_{q},R):=\left\{\exp_{q}(t\theta)|0\leq
t<R,~\theta\in\widehat{S}_{q}\right\},
\end{eqnarray*}
which is actually a geodesic cone with vertex at $q$. Using the
Type-I integral radial Ricci curvature
$k_{-}\left(p,q,\lambda(t),R\right)$ w.r.t. $q$ (see Definition
\ref{def1-1}), we can measure the volume of
$\Gamma(\widehat{S}_{q},R)$. In fact, we can prove:

\begin{theorem} \label{theorem3-10}
Given an $n$-dimensional ($n\geq2$) complete Riemannian manifold
$M$, assume that $p>\frac{n}{2}$ and that $\lambda(t)$ is a function
on $M$ satisfying the property \textbf{P1}. Then we have
\begin{eqnarray*}
\mathrm{vol}\left(\Gamma(\widehat{S}_{q},R)\right)\leq
G(n,p,\mathrm{vol}(\widehat{S}_{q}),\lambda_{\mathrm{inf}},k_{-}(q),R),
\end{eqnarray*}
where, similarly, $k_{-}(q):=k_{-}\left(p,q,\lambda(t),R\right)$,
$\lambda_{\inf}:=\inf_{(0,R)}\lambda(t)\leq0$, and
$\mathrm{vol}(\widehat{S}_{q})$ is the volume of $\widehat{S}_{q}$.
Moreover, as $\mathrm{vol}(\widehat{S}_{q})$,
$k_{-}(q)\rightarrow0$, we have
$$G(n,p,\mathrm{vol}(\widehat{S}_{q}),\lambda_{\mathrm{inf}},k_{-}(q),R)\rightarrow0.$$
\end{theorem}

\begin{proof}
For $x\in M\backslash\left(Cut(q)\cup\{q\}\right)$, define a
function $\widehat{\rho}$ as
$\widehat{\rho}(q,x):=\left|\min\left\{0,\mathrm{Ric}(v_{x},v_{x})\right\}\right|$,
which, under the situation $\lambda(t)\equiv0$, is exactly the
function $\rho(q,x)$ defined by (\ref{def1-1-F1}). Using relations
(\ref{FJ-1}), (\ref{FJ-2}), and a similar argument to that in the
proof of \cite[Lemma 7.3]{yd1}, for $(R,\theta)\in\mathcal{D}_{q}$,
we can get
\begin{eqnarray} \label{3-10-1}
\frac{J'-1}{J^{\widehat{\delta}}}(R,\theta)\leq\left(c_{6}(n,p)\int_{0}^{R}\widehat{\rho}^{p}J^{n-1}dt\right)^{\frac{1}{2p-1}},
\end{eqnarray}
where $p>\frac{n}{2}$, $\widehat{\delta}=\frac{2p-n}{2p-1}$, and
\begin{eqnarray} \label{def-c6}
c_{6}(n,p)=\left(2-\frac{1}{p}\right)^{p}\left(\frac{p-1}{2p-n}\right)^{p-1}.
\end{eqnarray}
Set
\begin{eqnarray*}
J_{+}(R,\theta):=\left\{
\begin{array}{ll}
J(R,\theta),  \qquad & (R,\theta)\in\mathcal{D}_{q},\\
0, \qquad & (R,\theta)\in T_{q}M\setminus\mathcal{D}_{q}
\end{array}
\right.
\end{eqnarray*}
and
\begin{eqnarray*}
\widehat{\mathcal{V}}(R):=\left(\int_{\widehat{S}_{q}}(J_{+}(R,\theta))^{n-1}d\theta\right)^{\frac{1}{n-1}}.
\end{eqnarray*}
By H\"{o}lder's inequality and (\ref{3-10-1}), it follows that
\begin{eqnarray} \label{3-10-2}
\widehat{\mathcal{V}}'(R)\leq
a_{1}+b_{1}\widehat{\mathcal{V}}^{\widehat{\delta}}(R), \qquad
\widehat{\mathcal{V}}(0)=0,
\end{eqnarray}
where
$a_{1}=\left(\mathrm{vol}(\widehat{S}_{q})\right)^{\frac{1}{n-1}}$
and
\begin{eqnarray*}
b_{1}=\left(c_{6}(n,p)\cdot
\left(k_{-}\left(p,q,0,R\right)\right)^p\right)^{\frac{1}{2p-1}}.
\end{eqnarray*}
Using a similar argument to that in the proof of Theorem
\ref{theorem3-9}, one can obtain
\begin{eqnarray*}
\int_{0}^{R}\left(\widehat{\mathcal{V}}(r)\right)^{n-1}dr=\mathrm{vol}\left(\Gamma(\widehat{S}_{q},R)\right)\leq\frac{\widehat{\alpha}\widehat{\gamma}+\widehat{\beta}}{\widehat{\gamma}^2}\left(e^{\widehat{\gamma}
R}-1\right)-\frac{\widehat{\beta}}{\widehat{\gamma}}R,
\end{eqnarray*}
where
\begin{eqnarray*}
&&
\widehat{\alpha}=a_{1}^{\frac{n-1}{n-\widehat{\delta}}}+2^{p-1}c_{6}(n,p)\cdot\left(\frac{2p-1}{2pb_{1}}\right)^{\frac{n-1}{n-\widehat{\delta}}}\cdot\left(k_{-}\left(p,q,\lambda(t),R\right)\right)^p,\\
&& \widehat{\beta}=a_{1}^{\frac{n-1}{n-\delta}},\\
&&\widehat{\gamma}=
a_{1}^{\frac{n-1}{n-\widehat{\delta}}}+\left(\frac{2p-1}{2p}\right)^{\frac{n-1}{n-\widehat{\delta}}}\cdot\left(2^{p-1}c_{6}(n,p)\cdot|\lambda_{\inf}|^{p}\right)^{\frac{1-\widehat{\delta}}{n-\widehat{\delta}}}.
\end{eqnarray*}
Theorem \ref{theorem3-10} follows naturally.
\end{proof}

\begin{remark}
\rm{ Clearly, since
$\Gamma(\widehat{S}_{q},R)=\left\{\exp_{q}(t\theta)|0\leq
t<R,~\theta\in\widehat{S}_{q}\right\}$, the geodesic
$\gamma(t)=\exp_{q}(t\theta)$ should be a unit-speed minimizing
geodesic, which implies that
\begin{eqnarray*}
0\leq R\leq\min\limits_{\theta\in\widehat{S}}d_{\theta},
\end{eqnarray*}
where $d_{\theta}$ is defined similarly as $d_{\xi}$ in Section
\ref{S2-P}.
 }
\end{remark}

At the end of this section, using a similar method to that of
\cite[Theorem 2]{sg}, we can obtain an upper bound for the volume of
normal tubes around hypersurfaces.

\begin{theorem}  \label{theorem3-13}
Assume that $N\subset M$ is a compact smoothly embedded hypersurface
of the given complete Riemannian $n$-manifold $M$, $n\geq2$, and $N$
satisfies the following regularity property:
\begin{itemize}

\item  For almost every $x\in
M\setminus N$, any minimizing geodesic from $x$ to $N$ attains $N$
at a regular point.

\end{itemize}
Besides, $\lambda(t)$ is a function on $M$ satisfying the property
\textbf{P2}. Then for every $p>\frac{n}{2}$ and any $R,s>0$,  we
have
\begin{eqnarray*}
&&\mathrm{vol}\left(\mathcal{T}(N,R)\right)\leq
\left(e^{c_{7}(n,p)\sqrt{|\lambda_{\mathrm{inf}}|}R}-1\right)\Bigg{[}\frac{2}{c_{7}(n,p)\sqrt{|\lambda_{\mathrm{inf}}|}}\mathrm{vol}(N)\\
&& \qquad \quad
+\frac{(n-1)^{2p-1}}{\left(c_{7}(n,p)\sqrt{|\lambda_{\mathrm{inf}}|}\right)^{2p}}
\int_{N}|H|^{2p-1}dv_{N} +
\left((n-1)\sqrt{|\lambda_{\mathrm{inf}}|}\right)^{-p}\cdot
\left(k_{-}\left(p,N,\lambda(t),R\right)\right)^{p}\Bigg{]},
\end{eqnarray*}
where, similarly, $\mathcal{T}(N,R)$ is defined as
(\ref{rem-EXTRA}), $\lambda_{\inf}:=\inf_{(0,R)}\lambda(t)\leq0$,
and $k_{-}\left(p,N,\lambda(t),R\right)$ is the Type-I integral
radial Ricci curvature w.r.t. $N$. Moreover, $H$ is the mean
curvature of $N$, which is given by
$H(z)=\frac{1}{n-1}\mathrm{Tr}(II_{z})$ for any $z\in N$ (i.e., the
trace of the second fundamental form $II_{z}$ at $z$), $dv_{N}$ is
the volume element of $N$ induced by the volume density $dv$ of $M$,
and
\begin{eqnarray*}
c_{7}(n,p):=\left(\frac{2p-1}{p}\right)^{\frac{1}{2}}(n-1)^{\frac{2p-1}{2p}}\left(\frac{2p-2}{2p-n}\right)^{\frac{p-1}{2p}}.
\end{eqnarray*}
\end{theorem}

\begin{proof}
Let $\Omega\subset M$ be any domain in $M$ whose boundary
$\partial\Omega$ satisfies the regularity property, and
$\Omega_{R}=\{x|d_{M}(x,\Omega)<R\}$. Let
$\widetilde{\partial\Omega}$ be the set of regular points in
$\partial\Omega$. Define a map $\Phi(t,x):
(-\infty,+\infty)\times\widetilde{\partial\Omega}\mapsto M$ given by
$\Phi(t,x)=\exp_{N}(t v_{x})$, where $v_{x}$ is the outward unit
normal vector at $x\in\widetilde{\partial\Omega}$, and, similar as
in (\ref{rem-EXTRA}), $\exp_{N}:\upsilon(N)\rightarrow M$ denotes
the normal exponential map. By Hopf-Rinow's theorem, since $M$ is
complete, one knows that $\Phi$ is surjective and should be a
diffeomorphism from some open subset
$U\subset(-\infty,+\infty)\times\widetilde{\partial\Omega}$ onto an
open subset $M^{\circ}\subset M$ whose complement is of Hausdorff
measure zero. This implies that there exists
$(t_{1},t_{2})\subset\mathbb{R}$ such that
$U=(t_{1},t_{2})\times\widetilde{\partial\Omega}$ and
$(t_{1},t_{2})$ is the greatest interval on which the geodesic
$\Phi(t,x)=\exp_{N}(t v_{x})$ minimizes the distance from
$\exp_{N}(t v_{x})$ to $\widetilde{\partial\Omega}$. Define a
function $L(t,x)$ by
\begin{eqnarray*}
\Phi^{\ast}dv=L^{n-1}(t,x)dtdx,
\end{eqnarray*}
where $\Phi^{\ast}$ is the pull-back induced by $\Phi$.
 Inspired by the derivation of (\ref{FJ-1}), we have, for any
$t\in(t_{1},t_{2})$, that
\begin{eqnarray*}
&&L''+\frac{1}{n-1}\mathrm{Ric}\left(\Phi_{\ast}(\frac{\partial}{\partial
t}),\Phi_{\ast}(\frac{\partial}{\partial
t})\right)L\leq0, \\
&& \qquad L(0,x)=1,\qquad L'(0,x)=H(x),
\end{eqnarray*}
where $\Phi_{\ast}(\cdot)$ is the tangential mapping induced by
$\Phi$ (one can see \cite{gro1,hk} for the detailed proof of the
above relations). Here we have used $(\cdot)'$ instead of
$\frac{\partial}{\partial t}(\cdot)$ for simplicity. By a direct
calculation, for any positive $\epsilon>0$ and at any point
$\Phi(t,x)$ with $(t,x)\in U$, one has
\begin{eqnarray} \label{3-13-1}
\left(\frac{L'}{L^{\epsilon}}\right)'+\epsilon\cdot\frac{(L')^2}{L^{1+\epsilon}}=\frac{L''}{L^{\epsilon}}\leq
h\cdot L^{1-\epsilon},
\end{eqnarray}
where
\begin{eqnarray*}
h(t):=\max\left\{0,-\frac{1}{n-1}\mathrm{Ric}\left(\Phi_{\ast}(\frac{\partial}{\partial
t}),\Phi_{\ast}(\frac{\partial}{\partial t})\right)\right\}.
\end{eqnarray*}
Since
\begin{eqnarray*}
p\left(\frac{p}{p-1}\right)^{p-1}y\leq\left(\max\{1+y,0\}\right)^{p}
\end{eqnarray*}
for any $y\in\mathbb{R}$, together with (\ref{3-13-1}), we have
 \begin{eqnarray*}
p\left(\frac{p\epsilon}{p-1}\right)^{p-1}\left(\frac{L'}{L^{\epsilon}}\right)'\left|\frac{L'}{L^{\epsilon}}\right|^{2(p-1)}\leq
h^{p}L^{(2p-1)(1-\epsilon)},
 \end{eqnarray*}
which, taking $\epsilon=\frac{2p-n}{2p-1}$ and integrating both
sides of the above inequality from $0$ to $t$, implies
\begin{eqnarray} \label{3-13-2}
\left(\frac{L'}{L^{\epsilon}}(t,x)\right)^{2p-1}\leq
\left(H_{+}(x)\right)^{2p-1}+2^{-(p-1)}(n-1)^{1-2p}c_{7}^{2p}(n,p)\cdot\int_{0}^{t}h^{p}(s)\cdot
L^{n-1}(s,x)ds.
\end{eqnarray}
Define
\begin{eqnarray*}
L_{+}(t,x):=\left\{
\begin{array}{ll}
L(t,x),  \qquad & t_{1}\leq t\leq t_{2},\\
0, \qquad & \mathrm{elsewhere},
\end{array}
\right.
\end{eqnarray*}
and
\begin{eqnarray*}
Q(R):=\int_{\partial\Omega}L_{+}^{n-1}(t,x)dx.
\end{eqnarray*}
Clearly, the volume of $\Omega_{R}\setminus\Omega$ is
\begin{eqnarray*}
\mathcal{A}(R):=\mathrm{vol}\left(\Omega_{R}\setminus\Omega\right)=\int_{\partial\Omega}\int_{0}^{R}L_{+}^{n-1}(t,x)dtdx.
\end{eqnarray*}
Using (\ref{3-13-2}), the H\"{o}lder inequality, the facts
\begin{eqnarray*}
&&\limsup\limits_{r\rightarrow0}\left[\frac{L_{+}(R+r,x)-L_{+}(R,x)}{r}\right]\leq\sup\{L'(R,x),0\},\\
&&\qquad Q'(R)=\limsup\limits_{r\rightarrow0}\frac{Q(R+r)-Q(R)}{r},
\end{eqnarray*}
and integrating over $\partial\Omega$, one has
\begin{eqnarray}  \label{3-13-3}
Q'\leq
Q^{\frac{2p-2}{2p-1}}\Bigg{[}(n-1)^{2p-1}\int_{\partial\Omega}\left(H_{+}(x)\right)^{2p-1}dx+2^{-(p-1)}c_{7}^{2p}(n,p)\cdot\int_{\Omega_{R}\setminus\Omega}
h^{p}dtdx\Bigg{]}^{\frac{1}{2p-1}},
\end{eqnarray}
where $H_{+}(x)=\max\{0,H(x)\}$. Together with the fact
 \begin{eqnarray*}
h^{p}\leq2^{p-1}\left[(\max\{h(t)+\lambda(t),0\})^{p}+|\lambda(t)|^{p}\right]\leq2^{p-1}\left[(\max\{h(t)+\lambda(t),0\})^{p}+|\lambda_{\mathrm{inf}}|^{p}\right],
 \end{eqnarray*}
it follows from (\ref{3-13-3}) that for any positive constant
$\tau>0$, one has
\begin{eqnarray}  \label{3-13-4}
&&\mathcal{A}(R+\tau)\leq
e^{c_{7}(n,p)\sqrt{|\lambda_{\mathrm{inf}}|}\tau}\cdot\mathcal{A}(R)+\left(e^{c_{7}(n,p)\sqrt{|\lambda_{\mathrm{inf}}|}\tau}-1\right)\cdot
\Bigg{[}\left(c_{7}(n,p)\sqrt{|\lambda_{\mathrm{inf}}|}\right)^{-1}\mathcal{A}'(0)\nonumber\\
&&\qquad \qquad\qquad
+\left(c_{7}(n,p)\sqrt{|\lambda_{\mathrm{inf}}|}\right)^{-2p}(n-1)^{2p-1}\int_{\partial\Omega}\left(H_{+}(x)\right)^{2p-1}dx+\nonumber\\
&&\qquad \qquad\qquad
\left((n-1)\sqrt{|\lambda_{\mathrm{inf}}|}\right)^{-p}
\int_{\Omega_{R+\tau}\setminus\Omega}|(n-1)\max\{h(t)+\lambda(t),0\}|^{p}dtdx\Bigg{]}.
\end{eqnarray}
Then the conclusion of Theorem \ref{theorem3-13} can be obtained
directly by applying (\ref{3-13-4}) to the domain
$\Omega=\cap_{\tau>0}\overline{\Omega_{\tau}}$ whose boundary is
made of two copies of $N$.
\end{proof}

\begin{remark}
\rm{ One can easily find that the method used in the proof of
Theorem \ref{theorem3-13} can be generalized to the case that the
codimension of $N$ is strictly greater than $1$. In fact, one only
needs to make some suitable adjustments to the above proof and then
an upper bound for the volume of normal tubes around
\emph{submanifolds} can be achieved. }
\end{remark}

We also have the following:

\begin{theorem} \label{theorem3-15}
Assume that $N\subset M$ is a compact smoothly embedded submanifold
of the given complete, noncompact Riemannian $n$-manifold $M$,
$n\geq2$, and $\lambda(t)$ is a function on $M$ satisfying the
property \textbf{P2} with $\lambda(t)\equiv0$. For any
$p>\frac{n}{2}$, we have:

(1) If $k_{-}\left(p,N,0\right)<\infty$, then
\begin{eqnarray*}
\lim\limits_{R\rightarrow\infty}\frac{\left(\mathrm{vol}\left(\mathcal{T}(N,R)\right)\right)^{\frac{1}{2p}}}{R}=
\lim\limits_{R\rightarrow\infty}\frac{\mathrm{vol}\left(\partial\mathcal{T}(N,R)\right)}{\left(\mathrm{vol}\left(\mathcal{T}(N,R)\right)\right)^{1-\frac{1}{2p}}}=0,
\end{eqnarray*}
where $k_{-}\left(p,N,0\right)$ is exactly the Type-I integral
radial Ricci curvature $k_{-}\left(p,N,\lambda(t),R\right)$ in
Theorem \ref{theorem3-13}, provided $\lambda(t)\equiv0$ and
$\mathcal{T}(N,R)$ covers $M$.

 (2) For any $s<\frac{1}{2p}$,
\begin{eqnarray*}
\limsup\limits_{R\rightarrow\infty}\frac{L_{s}\left(\mathrm{vol}\left(\mathcal{T}(N,R)\right)\right)}{R}&\leq&
\limsup\limits_{R\rightarrow\infty}\frac{\mathrm{vol}\left(\partial\mathcal{T}(N,R)\right)}{\left(\mathrm{vol}\left(\mathcal{T}(N,R)\right)\right)^{1-s}}\\
&\leq&c_{8}(n,p)\cdot\limsup\limits_{R\rightarrow\infty}\frac{\left(k_{-}\left(p,N,0,R\right)\right)^{\frac{1}{2}}}{\left(\mathrm{vol}\left(\mathcal{T}(N,R)\right)\right)^{\frac{1}{2p}-s}},
\end{eqnarray*}
where
\begin{eqnarray*}
L_{s}(x):=\left\{
\begin{array}{ll}
\frac{x^s}{s},  \qquad & \mathrm{if}~s\neq0,\\
\log(x), \qquad & \mathrm{if}~s=0,
\end{array}
\right.
\end{eqnarray*}
and
\begin{eqnarray*}
c_{8}(n,p):=2^{1-\frac{1}{2p}}(n-1)^{\frac{2p-1}{2p}}\left(\frac{p(p-1)}{(2p-1)(2p-n)}\right)^{\frac{p-1}{2p}}.
\end{eqnarray*}
\end{theorem}

\begin{proof}
Directly from (\ref{3-13-3}), we can obtain
\begin{eqnarray}  \label{3-15-1}
&&\left(\mathrm{vol}\left(\partial\Omega_{R}\right)\right)^{\frac{2p}{2p-1}}-\left(\mathrm{vol}\left(\partial\Omega\right)\right)^{\frac{2p}{2p-1}}\leq
2p(n-1)(2p-1)^{-1}\left(\mathrm{vol}\left(\Omega_{R}\right)-\mathrm{vol}\left(\Omega\right)\right)\nonumber\\
&& \qquad \times
\left[\int_{\partial\Omega}\left(H_{+}(x)\right)^{2p-1}dx+2^{-(p-1)}(n-1)^{-(2p-1)}c^{2p}_{7}(n,p)\cdot\int_{\Omega_{R}\setminus\Omega}h^{p}dtdx\right]^{\frac{1}{2p-1}}.
\qquad \qquad
\end{eqnarray}
Denote by $TM$ the tangent bundle of $M$.  Let $\Omega_{\epsilon}$
be the $\epsilon$-tubular neighborhood of $N$. Here $\epsilon>0$ is
chosen to be smaller than the injectivity radius of the exponential
map, which is defined on the subbundle of $TM$ over $N$ and
therefore which is normal to $N$. Clearly, $\Omega_{\epsilon}$ has
regular boundary. Using the fact that for a given function $F(t)$ on
$M$,
\begin{eqnarray*}
\limsup\limits_{R\rightarrow\infty}\frac{F(R)}{R}\leq\limsup\limits_{R\rightarrow\infty}F'(R),
\end{eqnarray*}
and applying (\ref{3-15-1}) to the domain
$\Omega=\Omega_{\epsilon}$, one has
\begin{eqnarray*}
\frac{\mathrm{vol}\left(\partial\mathcal{T}(N,R)\right)}{\left(\mathrm{vol}\left(\mathcal{T}(N,R)\right)\right)^{1-s}}\leq
c_{8}(n,p)\cdot\frac{\left(k_{-}\left(p,N,0,R\right)\right)^{\frac{1}{2}}}{\left(\mathrm{vol}\left(\mathcal{T}(N,R)\right)\right)^{\frac{1}{2p}-s}}+
O\left[\left(\mathrm{vol}\left(\mathcal{T}(N,R)\right)\right)^{s-\frac{1}{2p}}\right],
\end{eqnarray*}
which implies the conclusions of Theorem \ref{theorem3-15}.
\end{proof}

\begin{remark}
\rm{ (1) Especially, if the submanifold $N$ degenerates into a
single point $q\in M$, i.e., $N=\{q\}$, then the normal tube
$\mathcal{T}(N,R)$ becomes the geodesic ball $B(q,R)$, and in this
situation, (1) of Theorem \ref{theorem3-15} can be obtained by
directly applying the second conclusion of Theorem \ref{theorem2-1},
since, as explained clearly in (1) of Remark \ref{remark3-2}, the
volume of $B(q,R)$ \emph{at most} has  the growth $R^{n}$ provided
$k_{-}\left(p,q,0,R\right)$ is finite. \\
 (2) In fact, if $N=\{q\}$, then as shown in the proof of Theorem
 \ref{theorem3-15}, one has
 \begin{eqnarray*}
\frac{\mathrm{vol}(\partial B(q,R))}{\mathrm{vol}(B(q,R))}\leq
c_{8}(n,p)\cdot
\left(k_{-}\left(p,q,0,R\right)\right)^{\frac{1}{2}}\cdot\left(\mathrm{vol}(B(q,R))\right)^{-\frac{1}{2p}}+O\left[\left(\mathrm{vol}(B(q,R))\right)^{-\frac{1}{2p}}\right],
 \end{eqnarray*}
which gives an interesting isoperimetric property for the geodesic
ball $B(q,R)$ in terms of the Type-I integral radial Ricci curvature
w.r.t. $q$.
 }
\end{remark}

\section{Isoperimetric inequalities and their applications in spectral estimates}
\renewcommand{\thesection}{\arabic{section}}
\renewcommand{\theequation}{\thesection.\arabic{equation}}
\setcounter{equation}{0} \setcounter{maintheorem}{0}

In this section, first, we would like to mention two quantities
which have been introduced already in \cite{yd1}. For a given smooth
Riemannian $n$-manifold $M$, with the volume element $dv$,  and an
open set $U\subset M$, one can define \emph{local Sobolev constant}
$C_{S}(U)$ as the smallest number $A_1$ such that
\begin{eqnarray*}
\left(\int_{U}|h|^{\frac{2n}{n-2}}dv\right)^{\frac{n-2}{n}}\leq
A_{1}\int_{U}\|\nabla h\|^{2}dv, \qquad h\in C^{\infty}_{0}(U).
\end{eqnarray*}
The other quantity, named \emph{local isoperimetric constant}
$C_{I}(U)$, can be defined as the largest number $A_{2}>0$ such that
\begin{eqnarray*}
\int_{\partial\Omega}dv_{\partial\Omega}\geq
A_{2}\left(\int_{\Omega}dv\right)^{\frac{n-1}{n}} \qquad
\mathrm{i.e.}, ~\mathrm{vol}(\partial\Omega)\geq
A_{2}\cdot\left(\mathrm{vol}(\Omega)\right)^{\frac{n-1}{n}}
\end{eqnarray*}
holds for any compact domain $\Omega\subset U$ with smooth boundary,
where $dv_{\partial\Omega}$ is the volume element of
$\partial\Omega$ induced by $dv$. As shown in \cite{ic}, the
relation
\begin{eqnarray} \label{relation-1}
C_{S}(U)=4\left(\frac{n-1}{n-2}\right)^{2}\cdot\left(C_{I}(U)\right)^{-2}
\end{eqnarray}
holds.

Applying our volume estimate for geodesic cones (cf. Theorem
\ref{theorem3-10}), we can prove the following local isoperimetric
inequality if the Type-I integral radial Ricci curvature was assumed
to have an upper bound.

\begin{theorem} \label{theorem4-1}
Given an $n$-dimensional ($n\geq2$) complete Riemannian manifold
$M$, assume that $p>\frac{n}{2}$, $\tau>0$, and that $\lambda(t)$ is
a function on $M$ satisfying the property \textbf{P1} with
$\lambda(t)\equiv0$. Let
\begin{eqnarray*}
\eta=\left(\frac{B(q,r_1)}{\mathcal{B}_{n}(r_1)}\right)^{-\frac{1}{n}}=\left(\frac{B(q,r_1)}{n^{-1}w_{n}r_{1}^{n}}\right)^{-\frac{1}{n}},
\end{eqnarray*}
 \begin{eqnarray*}
r_2=\frac{\eta}{1+\tau}r_{1},
 \end{eqnarray*}
 and $r_{1}, r_{2}<\mathrm{inj}(q)$ with the injectivity radius $\mathrm{inj}(q)$ defined
 by (\ref{inj-R}). If
 \begin{eqnarray*}
c_{6}(n,p)\cdot r_{1}^{2p-n}\cdot\left(\sup\limits_{x\in M}
k_{-}(p,x,0,r_{1}+2r_{2})\right)^{p}\leq
w_{n}\cdot\min\left\{\tau^{2p-1}\eta^{n-2p},\frac{2p(n-1)}{n(2p-1)}\frac{\tau\eta^{n}}{(1+\tau+\eta)^{2p}}\right\}
 \end{eqnarray*}
 with $c_{6}(n,p)$ given by (\ref{def-c6}), then
 \begin{eqnarray*}
 C_{I}(B(q,r_{2}))\geq 2^{\frac{n-1}{n}} \cdot\left(\frac{\tau\eta}{1+\tau+\eta}\right)^{n+1}C_{I}(\mathbb{R}^n).
 \end{eqnarray*}
\end{theorem}

\begin{proof}
Define
\begin{eqnarray*}
\mathcal{W}(R):=\left\{
\begin{array}{ll}
(1+\tau)a_{1}R, \qquad & 0\leq R\leq R_{0},\\
\left[(1+\tau^{-1})(1-\widehat{\delta})b_{1}R+\widehat{\delta}(\tau
a_{1}b_{1}^{-1})^{(1-\widehat{\delta})/\widehat{\delta}}\right]^{1/(1-\widehat{\delta})},
\qquad & R\geq R_0,
\end{array}
\right.
\end{eqnarray*}
where
$a_{1}=\left(\mathrm{vol}(\widehat{S}_{q})\right)^{\frac{1}{n-1}}$,
$b_{1}=\left(c_{6}(n,p)\cdot
\left(k_{-}\left(p,q,0,R\right)\right)^p\right)^{\frac{1}{2p-1}}$
defined as in (\ref{3-10-2}), and $R_0$ is given as
\begin{eqnarray*}
 R_{0}=\frac{1}{(1+\tau)a_{1}}\left(\frac{\tau a_{1}}{b_{1}}\right)^{1/\widehat{\delta}}
\end{eqnarray*}
with $\widehat{\delta}=\frac{2p-n}{2p-1}$. By direct calculation,
one can easily have
\begin{eqnarray*}
\mathcal{W}'(R)\geq a_{1}+b_{1}\mathcal{W}^{\widehat{\delta}}(R).
\end{eqnarray*}
Combining the above differential inequality with (\ref{3-10-2}), we
have $\widehat{\mathcal{V}}(R)\leq\mathcal{W}(R)$ for any $0\leq
R<\mathrm{inj}(q)$, and then the volume of the geodesic cone
$\Gamma(\widehat{S}_{q},R)$ can be estimated as follows
\begin{eqnarray*}
\mathrm{vol}\left(\Gamma(\widehat{S}_{q},R)\right)&=&\int_{0}^{R}\left(\widehat{\mathcal{V}}(r)\right)^{n-1}dr\leq\int_{0}^{R}\left(\widehat{\mathcal{W}}(r)\right)^{n-1}dr\\
&=& (1+\tau)^{n-1}n^{-1}\mathrm{vol}(\widehat{S}_{q})R^{n}, \qquad
\qquad \mathrm{if}~0\leq R\leq R_{0},
\end{eqnarray*}
or
\begin{eqnarray*}
\mathrm{vol}\left(\Gamma(\widehat{S}_{q},R)\right)&=&\frac{\tau^{n/\widehat{\delta}}}{1+\tau}\left[\frac{1}{n}-\frac{2p-1}{2p(n-1)}\right]a_{1}^{(n/\widehat{\delta})-1}b_{1}^{-n/\widehat{\delta}}\\
&&\qquad +c_{9}(n,p,\tau)\cdot
\int_{\Gamma(\widehat{S}_{q},R)}\rho^{p}dv\cdot\left[(1-\widehat{\delta})R+\widehat{\delta}R_{0}\right]^{2p}\\
&\leq&c_{9}(n,p,\tau)\cdot
\left(k_{-}(p,q,0,R)\right)^{p}\left[(1-\widehat{\delta})R+\widehat{\delta}R_{0}\right]^{2p},
\qquad \mathrm{if}~R_{0}\leq R\leq\mathrm{inj}(q),
\end{eqnarray*}
where
\begin{eqnarray*}
c_{9}(n,p,\tau):=\left(1+\frac{1}{\tau}\right)^{2p-1}\cdot\frac{2p-1}{2p(n-1)}\cdot
c_{6}(n,p).
\end{eqnarray*}
For any compact domain $\Omega\subset B(q,r_{2})$ with smooth
boundary $\partial\Omega$, $x\in\Omega$, denote by
$\widehat{S}_{x}\subset S_{x}M$ the set of unit tangent vectors $v$
such that the geodesic $\gamma(s)=\exp_{x}(sv)$ is a minimizing
geodesic joining $x$ and some point in $B(q,r_{2})\setminus
B(q,r_{1})$. Choose $x\in\Omega$ such that $\widehat{S}_{x}$ has
minimal volume. By \cite[Theorem 11]{cbc}, one has
\begin{eqnarray} \label{4-1-1}
\frac{\mathrm{vol}(\partial\Omega)}{\left(\mathrm{vol}(\Omega)\right)^{\frac{n-1}{n}}}\geq
2^{\frac{n-1}{n}} \cdot
C_{I}(\mathbb{R}^n)\cdot\left(\frac{\mathrm{vol}(\widehat{S}_{x})}{w_{n}}\right)^{1+\frac{1}{n}},
\end{eqnarray}
where $C_{I}(\mathbb{R}^n)=w_{n}/w_{n+1}^{(n-1)/n}$ is the local
isoperimetric constant of $\mathbb{R}^n$. Under the assumption for
Type-I integral radial Ricci curvature in Theorem \ref{theorem4-1},
together with the above upper bounds for geodesic cones, one has
\begin{eqnarray} \label{4-1-2}
\mathrm{vol}(B(q,r_2))\leq (1+\tau)^{n-1}n^{-1}w_{n}r_{2}^{n}.
\end{eqnarray}
This implies that
\begin{eqnarray} \label{4-1-3}
\mathrm{vol}\left(\Gamma(\widehat{S}_{x},r_{1}+r_{2})\right)\geq\mathrm{vol}(B(q,r_{1})\setminus
B(q,r_{2}))\geq n^{-1}w_{n}\frac{\tau(\eta r_{1})^n}{1+\tau}.
\end{eqnarray}
On the other hand, one can easily check
 \begin{eqnarray*}
\mathrm{vol}\left(\Gamma(\widehat{S}_{x},r_{1}+r_{2})\right)\geq
c_{9}(n,p,\tau^{-1})\cdot
\left(\int_{\Gamma(\widehat{S}_{x},R)}\rho^{p}dv\right)^{\frac{1}{2p}}\cdot(r_{1}+r_{2})^{2p}
\end{eqnarray*}
holds, which implies
\begin{eqnarray} \label{4-1-4}
\mathrm{vol}\left(\Gamma(\widehat{S}_{x},r_{1}+r_{2})\right)\leq(1+\tau^{-1})^{n-1}n^{-1}\mathrm{vol}(\widehat{S}_{x})\cdot(r_{1}+r_{2})^{n}
\end{eqnarray}
 if the supremum of the Type-I integral radial Ricci curvature
was assumed to have an upper bound shown in Theorem
\ref{theorem4-1}. Combining (\ref{4-1-3}) and  (\ref{4-1-4}) yields
\begin{eqnarray*}
\frac{\mathrm{vol}(\widehat{S}_{x})}{w_{n}}\geq\left(\frac{\tau\eta}{1+\tau+\eta}\right)^{n+1}.
\end{eqnarray*}
Our estimate for $C_{I}(B(q,r_{2}))$ follows by substituting the
above inequality into (\ref{4-1-1}) directly.
\end{proof}

\begin{remark} \label{remark4-2}
\rm{ (1) Together with (\ref{relation-1}), one can get that under
the assumptions of Theorem \ref{theorem4-1}, the local Sobolev
constant $C_{S}(B(q,r_{2}))$ satisfies
\begin{eqnarray*}
C_{S}(B(q,r_{2}))&\leq&2^{-\frac{2}{n-1}}\left(\frac{n-1}{n-2}\right)^{2}\left(\frac{1+\tau+\eta}{\tau\eta}\right)^{-2(n+1)}\left(C_{I}(\mathbb{R}^n)\right)^{-2}\\
&=&2^{-\frac{2}{n-1}}\left(\frac{n-1}{n-2}\right)^{2}\left(\frac{1+\tau+\eta}{\tau\eta}\right)^{-2(n+1)}\cdot\frac{w_{n+1}^{\frac{2(n-1)}{n}}}{w_{n}^{2}}.
\end{eqnarray*}
 (2) Clearly, if one imposes an upper bound assumption for
 $k_{-}(p,q,0,r_2)$,
 the volume estimate (\ref{4-1-2}) for the geodesic ball $B(q,r_2)$ can be also
 obtained directly from the second assertion of Theorem \ref{theorem2-1}. \\
 (3) From Theorem \ref{theorem3-10}, we know that the upper bound
 estimate for volume of geodesic cones therein is still valid for the
 case that $\lambda(t)$ is not a constant function. Besides, \cite[Theorem
 11]{cbc} works for any smooth compact manifold with smooth
 boundary. Therefore, using a similar argument to that in the proof
 of Theorem \ref{theorem4-1}, an upper bound, which would involve a quantity from the spherically symmetric $n$-manifold $M^{-}:=[0,l)\times_{f}\mathbb{S}^{n-1}$, for
  $C_{S}(B(q,r_{2}))$ can be expected, where $f(t)$ is determined by
  the system (\ref{ODE}).
 }
\end{remark}

Now, we would like to use Theorem \ref{theorem3-13} to get an
interesting isoperimetric inequality. However, before that we need
to use a quantity introduced in \cite{sg}.

Given a Riemannian $n$-manifold  $M$ whose volume $\mathrm{vol}(M)$
is finite, $n\geq2$, then, as in \cite[Lemma 4]{sg}, one can define
an isoperimetric quantity $\mathrm{Is}(p)$ for $M$ as follows
 \begin{eqnarray} \label{IQ-D}
\mathrm{Is}(p):=\inf\left\{\frac{\mathrm{vol}(\partial\Omega)}{\left(\mathrm{vol}(\Omega)\right)^{1-\frac{1}{2p}}}\cdot\left(\mathrm{vol}(M)\right)^{-\frac{1}{2p}}
\Bigg{|}\Omega\subset
M,~\mathrm{vol}(\Omega)\leq\frac{\mathrm{vol}(M)}{2}\right\}.
 \end{eqnarray}
Gallot \cite[Lemma 4]{sg} proved that

\begin{lemma} \label{lemma4-3}
For any $p>\frac{n}{2}$, there exists a minimal current $\Omega_{0}$
in $M$ such that
\begin{eqnarray*}
\frac{\mathrm{vol}(\partial\Omega_0)}{\left(\mathrm{vol}(\Omega_0)\right)^{1-\frac{1}{2p}}}\cdot\left(\mathrm{vol}(M)\right)^{-\frac{1}{2p}}=\mathrm{Is}(p).
\end{eqnarray*}
Besides, this current has the following properties:

(1) For almost every point $x$ in $M$, any geodesic of minimal
length from $x$ to $\partial\Omega_{0}$ reaches $\partial\Omega_{0}$
at a regular point $x'$. Moreover, there exists a neighborhood $U$
of $x'$ in $M$ such that $U\cap\partial\Omega_{0}$ is smooth.

(2) Let us call $\widetilde{\partial\Omega_{0}}$ the set of all
regular points of $\partial\Omega_{0}$. The mean curvature $H$ of
$\widetilde{\partial\Omega_{0}}$ is constant and satisfies
$|H|\leq\frac{2p-1}{2p(n-1)}\cdot\frac{\mathrm{vol}(\partial\Omega_{0})}{\mathrm{vol}(\Omega_{0})}$.
Moreover, if $\mathrm{vol}(\Omega_0)\neq\frac{\mathrm{vol}(M)}{2}$,
then
$H=\frac{2p-1}{2p(n-1)}\cdot\frac{\mathrm{vol}(\partial\Omega_{0})}{\mathrm{vol}(\Omega_{0})}$.
\end{lemma}

\begin{remark}
\rm{ It is easy to see that (\ref{IQ-D}) can also be rewritten as
follows
\begin{eqnarray*}
\mathrm{Is}(p)=\inf\left\{\frac{\mathrm{vol}(\partial\Omega)}{\left(\min\{\mathrm{vol}(\Omega),\mathrm{vol}(M\setminus\Omega)\}\right)^{1-\frac{1}{2p}}}\cdot\left(\mathrm{vol}(M)\right)^{-\frac{1}{2p}}
\Bigg{|}\Omega\subset
M,~\partial\Omega~{\mathrm{is~regular}}\right\}.
\end{eqnarray*}
}
\end{remark}

We can prove:

\begin{theorem} \label{theorem4-4}
In any $n$-dimensional ($n\geq2$) Riemannian manifold $M$ whose
integral radial Ricci curvature satisfies
\begin{eqnarray} \label{4-4-1}
\left((n-1)\sqrt{|\lambda_{\mathrm{inf}}|}\right)^{-p}\cdot
\left(k_{-}\left(p,\partial\Omega_0,\lambda(t)\right)\right)^{p}\leq\frac{\mathrm{vol}(M)}{2}\cdot\left(e^{c_{7}(n,p)\sqrt{|\lambda_{\mathrm{inf}}|}D}-1\right)^{-1}
\end{eqnarray}
for $p>\frac{n}{2}$ and any positive constant $D$ such
that\footnote{ The precondition (\ref{4-4-2}) is used to make sure
that the normal tube $\mathcal{T}(\partial\Omega_0,D)$ covers the
manifold $M$ completely.}
\begin{eqnarray}  \label{4-4-2}
\sup\limits_{q\in\widetilde{\partial\Omega_0}}\ell(q)\leq D,
\end{eqnarray}
then every domain $\Omega\subset M$, with regular boundary,
satisfies
\begin{eqnarray*}
\frac{\mathrm{vol}(\partial\Omega)}{\mathrm{vol}(M)}\geq
c_{10}(n,p,\lambda_{\mathrm{inf}},D)\cdot\left(\inf\left\{\frac{\mathrm{vol}(\Omega)}{\mathrm{vol}(M)},\frac{\mathrm{vol}(M\setminus\Omega)}{\mathrm{vol}(M)}\right\}\right)^{1-\frac{1}{2p}},
\end{eqnarray*}
where
\begin{eqnarray*}
c_{10}(n,p,\lambda_{\mathrm{inf}},D):=c_{7}(n,p)\sqrt{|\lambda_{\mathrm{inf}}|}\cdot\min\left\{2^{-\frac{1}{2p-1}},
\frac{1}{4}\left(e^{c_{7}(n,p)\sqrt{|\lambda_{\mathrm{inf}}|}D}-1\right)^{-1}\right\},
\end{eqnarray*}
and, as in Lemma \ref{lemma4-3}, $\Omega_0$ is the minimal current
in $M$, $\widetilde{\partial\Omega_0}$ is the set of all regular
points of $\Omega_0$. Besides, $\lambda(t)$ is a function on $M$
satisfying the property \textbf{P2} (with $N=\Omega_0$ here),
$\ell(q)$ is defined as (\ref{key-def1}),
$\lambda_{\inf}:=\inf_{(0,D)}\lambda(t)\leq0$, $c_{7}(n,p)$ is
defined as in Theorem \ref{theorem3-13}, and, similarly,
 \begin{eqnarray*}
k_{-}\left(p,\partial\Omega_0,\lambda(t)\right)&=&\left(\int_{\mathcal{T}(\partial\Omega_0,D)}\left|\min\left\{0,\mathrm{Ric}(v_{x},v_{x})-(n-1)\lambda(t)\right\}\right|^{p}dv\right)^{\frac{1}{p}}\\
&=&\left(\int_{M}\left|\min\left\{0,\mathrm{Ric}\left(\frac{d}{dt}\Big{|}_x,\frac{d}{dt}\Big{|}_{x}\right)-(n-1)\lambda(t)\right\}\right|^{p}dv\right)^{\frac{1}{p}}.
\end{eqnarray*}
\end{theorem}

\begin{proof}
Combining Lemma \ref{lemma4-3} and Theorem \ref{theorem3-13}, one
can obtain
\begin{eqnarray*}
&&\mathrm{vol}(M)\leq
\left(e^{c_{7}(n,p)\sqrt{|\lambda_{\mathrm{inf}}|}D}-1\right)\Bigg{[}\frac{2}{c_{7}(n,p)\sqrt{|\lambda_{\mathrm{inf}}|}}\mathrm{vol}(\partial\Omega_0)\\
&& \qquad \qquad \qquad
+\left(\frac{2p-1}{2p}\right)^{2p-1}\left(\frac{\mathrm{Is}(p)}{c_{7}(n,p)\sqrt{|\lambda_{\mathrm{inf}}|}}\right)^{2p}
\mathrm{vol}(M)\\
&& \qquad \qquad \qquad +
\left((n-1)\sqrt{|\lambda_{\mathrm{inf}}|}\right)^{-p}\cdot
\left(k_{-}\left(p,N,\lambda(t),R\right)\right)^{p}\Bigg{]},
\end{eqnarray*}
which, together with the assumption (\ref{4-4-1}), implies
\begin{eqnarray*}
&&\frac{\mathrm{vol}(M)}{2}\leq
\left(e^{c_{7}(n,p)\sqrt{|\lambda_{\mathrm{inf}}|}D}-1\right)\Bigg{[}\frac{2}{c_{7}(n,p)\sqrt{|\lambda_{\mathrm{inf}}|}}\mathrm{vol}(\partial\Omega_0)\\
&& \qquad \qquad \qquad
+\left(\frac{2p-1}{2p}\right)^{2p-1}\left(\frac{\mathrm{Is}(p)}{c_{7}(n,p)\sqrt{|\lambda_{\mathrm{inf}}|}}\right)^{2p}
\mathrm{vol}(M)\Bigg{]}.
\end{eqnarray*}
Putting the fact
$\mathrm{vol}(\Omega_0)\leq\frac{\mathrm{vol}(M)}{2}$ into the above
inequality yields $\mathrm{Is}(p)\geq
c_{10}(n,p,\lambda_{\mathrm{inf}},D)$, which implies the conclusion
of Theorem \ref{theorem4-4} directly.
\end{proof}

 At the end of this section, by applying Theorem \ref{theorem3-15},
we would like to give a nice sharper estimate for the infimum of the
spectrum of the Laplacian $\Delta$ on a complete noncompact
manifold.

Given an $n$-dimensional ($n\geq2$) complete noncompact manifold $M$
with the metric $g$, denote by
$\inf\left(\mathrm{spec}(\Delta,M,g)\right)$ the infimum of the
spectrum of $\Delta$ on $(M,g)$.  Donnelly \cite{hdo} proved that
 \begin{eqnarray} \label{d-4-1}
\inf\left(\mathrm{spec}(\Delta,M,g)\right)\leq\frac{(n-1)^2}{4}\|r_{-}\|_{L^{\infty}},
 \end{eqnarray}
where
 \begin{eqnarray*}
 r_{-}(x):=\sup\left\{0,-\inf\limits_{X\in
 T_{x}M\setminus\{0\}}\frac{\mathrm{Ric}(X,X)}{(n-1)g(X,X)}\right\}.
 \end{eqnarray*}
Clearly, if
$\mathrm{Ric}(\cdot,\cdot)\geq-(n-1)\alpha^{2}g(\cdot,\cdot)$ for
some constant $\alpha$, i.e., $\|r_{-}\|_{L^{\infty}}=\alpha^2$,
then one has
$$\inf\left(\mathrm{spec}(\Delta,M,g)\right)\leq\frac{(n-1)^2}{4}\alpha^{2}.$$
Gallot \cite[Proposition 7]{sg} extended Donnelly's conclusion
(\ref{d-4-1}) to the situation that the upper bound can be given in
terms of the norm $\|r_{-}\|_{L^{p/2}}$, $p>n$. Besides, he also
showed that this upper bound estimate is \emph{sharp} for the
hyperbolic space $\mathbb{H}^{n}(-\alpha^2)$, which coincides with
Mckean's result \cite{hpm} that the spectrum of $\Delta$ on
$\mathbb{H}^{n}(-\alpha^2)$ is
$\mathrm{spec}(\Delta,\mathbb{H}^{n}(-\alpha^2),\cdot)=[\frac{(n-1)^2}{4}\alpha^{2},\infty)$.
The above Donnelly's and Mckean's results can be improved to a more
general setting. In fact, we can prove:

\begin{theorem}  \label{theorem4-5}
Assume that $M$ is an $n$-dimensional ($n\geq2$) complete
$n$-manifold with the metric $g$ and non finite volume, $q\in M$,
and $\lambda(t)$ is a function on $M$ satisfying the property
\textbf{P1}. For any $p>\frac{n}{2}$, we have
\begin{eqnarray} \label{4-5-1}
\inf\left(\mathrm{spec}(\Delta,M,g)\right)&\leq&\frac{1}{4}\left(\limsup\limits_{R\rightarrow\infty}\frac{\mathrm{vol}(\partial
B(q,R))}{\mathrm{vol}(B(q,R))}\right)^{2}\nonumber\\
&\leq&
\left(\frac{c_{8}(n,p)}{2\sqrt{n-1}}\right)^{2}\limsup\limits_{R\rightarrow\infty}\overline{k_{-}}\left(p,q,0,R\right),
\end{eqnarray}
where $c_{8}(n,p)$ is given in Theorem \ref{theorem3-15}, and,
similar as in Definition \ref{def1-1},
$\overline{k_{-}}\left(p,q,0,R\right)$ is the average of the Type-I
integral radial Ricci curvature $k_{-}\left(p,q,0,R\right)$ w.r.t.
$q$. In particular, if
 \begin{eqnarray*}
k_{-}\left(p,q,\alpha^{2}\right):&=&\limsup\limits_{R\rightarrow\infty}k_{-}\left(p,q,-\alpha^{2},R\right)\\
&=&
\limsup\limits_{R\rightarrow\infty}\left(\int_{B(q,R)}\left|\min\left\{0,\mathrm{Ric}(v_{x},v_{x})+(n-1)\alpha^{2}\right\}\right|^{p}dv\right)^{\frac{1}{p}}\\
&=&\left(\int_{M}\left|\min\left\{0,\mathrm{Ric}(v_{x},v_{x})+(n-1)\alpha^{2}\right\}\right|^{p}dv\right)^{\frac{1}{p}}
 \end{eqnarray*}
 is finite, i.e., $k_{-}\left(p,q,\alpha^{2}\right)<\infty$, then we
 have
  \begin{eqnarray} \label{4-5-1-1}
\inf\left(\mathrm{spec}(\Delta,M,g)\right)\leq\left(\frac{c_{8}(n,p)}{2}\right)^{2}\alpha^{2}.
  \end{eqnarray}
\end{theorem}

\begin{proof}
By Rayleigh's theorem and the Max-min principle, we know that for
any function
$u(t)=u\left(d_{M}(q,\cdot)\right):(0,\infty)\mapsto\mathbb{R}$ with
\begin{eqnarray*}
\int_{0}^{\infty}\left[\left(u'(t)\right)^{2}+\left(u(t)\right)^{2}\right]\cdot\mathrm{vol}(\partial
B(q,t))dt<\infty,
\end{eqnarray*}
one has $u(t)\in W^{1,2}(M)$ and
\begin{eqnarray} \label{4-5-2}
\inf\left(\mathrm{spec}(\Delta,M,g)\right)\leq\frac{\int_{0}^{\infty}\left(u'(t)\right)^{2}\cdot\mathrm{vol}(\partial
B(q,t))dt}{\int_{0}^{\infty}u^{2}(t)\cdot\mathrm{vol}(\partial
B(q,t))dt},
\end{eqnarray}
where $W^{1,2}(M)$ is the completion of the set $C^{\infty}(M)$ of
smooth functions under the Sobolev norm
$\|\phi\|_{1,2}(M):=\left[\int_{M}\left(\|\nabla\phi\|^{2}+\phi^{2}\right)dv\right]^{1/2}$.

Set
\begin{eqnarray*}
c_{11}=\limsup\limits_{R\rightarrow\infty}\frac{\mathrm{vol}(\partial
B(q,R))}{\mathrm{vol}(B(q,R))}.
\end{eqnarray*}
For any positive constant $\epsilon>0$, there exists some
$R_{0}=R_{0}(\epsilon)$ such that for any $R\geq R_{0}$,
\begin{eqnarray*}
\mathrm{vol}(\partial
B(q,R))\leq(c_{11}+\epsilon)\cdot\mathrm{vol}(B(q,R)).
\end{eqnarray*}
Integrating the above inequality over $[R_0,R)$ gives
\begin{eqnarray*}
\mathrm{vol}(B(q,R))\leq
e^{(c_{11}+\epsilon)(R-R_{0})}\mathrm{vol}(B(q,R_{0})),
\end{eqnarray*}
which implies
\begin{eqnarray*}
\mathrm{vol}(\partial
B(q,R))\leq(c_{11}+\epsilon)e^{(c_{11}+\epsilon)(R-R_{0})}\mathrm{vol}(B(q,R_{0})).
\end{eqnarray*}
Since
\begin{eqnarray*}
&&\qquad
\int_{0}^{\infty}\left[\left(\frac{c_{11}+2\epsilon}{2}\cdot
e^{-(c_{11}+2\epsilon)t/2}\right)^{2}+\left(e^{-(c_{11}+2\epsilon)t/2}\right)^{2}\right]\cdot\mathrm{vol}(\partial
B(q,t))dt\\
&&=\left[\frac{(c_{11}+2\epsilon)^2}{4}+1\right]\cdot\int_{0}^{\infty}e^{-(c_{11}+2\epsilon)t}\cdot\mathrm{vol}(\partial
B(q,t))dt\\
&&=
\left[\frac{(c_{11}+2\epsilon)^2}{4}+1\right]\cdot\Bigg{[}\int_{0}^{R_0}e^{-(c_{11}+2\epsilon)t}\cdot\mathrm{vol}(\partial
B(q,t))dt+\int_{R_0}^{\infty}e^{-(c_{11}+2\epsilon)t}dt \\
 && \qquad \times
(c_{11}+\epsilon)e^{(c_{11}+\epsilon)(R-R_{0})}\mathrm{vol}(B(q,R_{0}))\Bigg{]}\\
&&<\infty,
\end{eqnarray*}
the function $u(t)=e^{-(c_{11}+2\epsilon)t/2}$, $t>0$, belongs to
the Sobolev space $W^{1,2}(M)$. Substituting
$u(t)=e^{-(c_{11}+2\epsilon)t/2}$ into (\ref{4-5-2}) results into
\begin{eqnarray*}
\inf\left(\mathrm{spec}(\Delta,M,g)\right)\leq\frac{(c_{11}+2\epsilon)^2}{4},
\end{eqnarray*}
which, by letting $\epsilon\rightarrow0$, implies
\begin{eqnarray*}
\inf\left(\mathrm{spec}(\Delta,M,g)\right)\leq\frac{c_{11}^{2}}{4}=\frac{1}{4}\left(\limsup\limits_{R\rightarrow\infty}\frac{\mathrm{vol}(\partial
B(q,R))}{\mathrm{vol}(B(q,R))}\right)^{2}.
\end{eqnarray*}
Together with (2) of Theorem \ref{theorem3-15}, where $s=0$, one can
get the second inequality in the expression (\ref{4-5-1}) directly.

Since
 \begin{eqnarray*}
&&\limsup\limits_{R\rightarrow\infty}\overline{k_{-}}\left(p,q,0,R\right)=\limsup\limits_{R\rightarrow\infty}\left(\int_{B(q,R)}\left|\min\left\{0,\mathrm{Ric}(v_{x},v_{x})\right\}\right|^{p}dv\right)^{\frac{1}{p}}\\
&& \qquad
\leq\limsup\limits_{R\rightarrow\infty}\left[\frac{1}{\mathrm{vol}(B(q,R))}\int_{B(q,R)}\left(\left|\min\left\{0,\mathrm{Ric}(v_{x},v_{x})+(n-1)\alpha^{2}\right\}\right|
+(n-1)\alpha^{2}\right)^{p}dv\right]^{\frac{1}{p}}\\
&&\qquad =(n-1)\alpha^{2},
 \end{eqnarray*}
then the second assertion of Theorem \ref{theorem4-5} follows by
substituting the above inequality into (\ref{4-5-1}) directly.
\end{proof}

\begin{remark} \label{remark4-6}
\rm{ (1) Since $\lim_{p\rightarrow\infty}c_{8}(n,p)=n-1$, our upper
bound estimate (\ref{4-5-1-1}) is sharp for the hyperbolic $n$-space
$\mathbb{H}^{n}(-\alpha^2)$. \\
 (2)  Let
$\lambda_{1}^{D}(B(q,R))$ be the first Dirichlet eigenvalue of the
Laplacian $\Delta$ on $B(q,R)\subset M$. By domain monotonicity of
eigenvalues with vanishing Dirichlet data (cf. \cite[pages
17-18]{ic}), one knows that $\lambda_{1}^{D}(B(q,R))$ decreases as
$R$ increases, and moreover, one can define the limit
 \begin{eqnarray*}
  \lambda_{1}(M):=\lim\limits_{R\rightarrow\infty}\lambda_{1}^{D}(B(q,R)),
  \end{eqnarray*}
which is independent of the choice of the point $q\in M$. In
general, $\lambda_{1}(M)\geq0$. Schoen and Yau \cite[page 106]{sy}
suggested that it is an important question to find conditions such
that $\lambda_{1}(M)>0$. Speaking in other words, open manifolds
with $\lambda_{1}(M)>0$ might have some special geometric
properties. There are many interesting results supporting this. For
instance,
 Mckean \cite{hpm}
showed that for an $n$-dimensional complete noncompact, simply
connected Riemannian manifold $M$ with sectional curvature
$K\leq-\alpha^{2}<0$,
$\lambda_{1}(M)\geq\frac{(n-1)^{2}\alpha^{2}}{4}>0$, and moreover,
$\lambda_{1}(\mathbb{H}^{n}(-\alpha^{2}))=\frac{(n-1)^{2}\alpha^{2}}{4}$.
Grigor'yan \cite{ag} showed that if $\lambda_{1}(M)>0$, then $M$ is
non-parabolic, i.e., there exists a non-constant bounded subharmonic
function on $M$. Cheung and Leung \cite{cl} proved that if $M$ is an
$n$-dimensional complete minimal submanifold in the hyperbolic
$m$-space $\mathbb{H}^{m}(-1)$, then
$\lambda_{1}(M)\geq\frac{(n-1)^2}{4}>0$, and moreover, $M$ is
non-parabolic. They also showed that if furthermore $M$ has at least
two ends, then there exists on $M$ a non-constant bounded harmonic
function with finite Dirichlet energy.

For the complete noncompact $n$-manifold $M$, the spectrum of
$\Delta$ on $M$ is complicated. It might only have the essential
spectrum or have both the essential spectrum and the discrete
spectrum. For instance, $\Delta$ on $\mathrm{H}^{n}(-\alpha^2)$ or
$\mathbb{R}^n$ only has the essential spectrum, while $\Delta$ on
the $4$-dimensional rotationally symmetric quantum layer constructed
in \cite{m2013} has both the essential spectrum and the discrete
spectrum. Besides, it is easy to know that
$\inf\left(\mathrm{spec}(\Delta,M,g)\right)\leq\lambda_{1}(M)$, and
the equality holds if $\Delta$ only has the essential spectrum.

By using Theorem \ref{theorem4-5}, the fact
$\inf\left(\mathrm{spec}(\Delta,M,g)\right)\leq\lambda_{1}(M)$, and
Cheng's eigenvalue comparison theorems \cite{cheng1} (CECTs for
short),\footnote{ CECTs \cite{cheng1,cheng2} have been extended to
complete manifolds with radial (Ricci or sectional) curvature
bounded (from below or from above) by Freitas, Mao and Salavessa --
see \cite[Theorems 3.6 and 4.4]{fmi} for details. For the case of
the nonlinear $\flat$-Laplacian, Mao also successfully obtained a
Cheng-type eigenvalue comparison for manifolds with radial Ricci
curvature bounded from below -- see \cite[Theorem 3.2]{m2} for
details. However, for manifolds with radial sectional curvature
bounded from above, a Cheng-type eigenvalue comparison might not be
obtained, since Barta's lemma (cf. \cite{JB}) cannot be transplanted
to the case of the nonlinear $\flat$-Laplace operator directly --
for detailed explanation, see the last paragraph of \cite[Subsection
3.2, Chap. 3]{m1}.} we have:

\begin{itemize}

\item \emph{For a given complete noncompact $n$-manifold $(M,g)$, $n\geq2$, if
$\mathrm{Ric}(\cdot,\cdot)\geq-(n-1)\alpha^{2}g(\cdot,\cdot)$ and
$K\leq-\alpha^2$ for some nonzero constant $\alpha$, then
 \begin{eqnarray*}
\inf\left(\mathrm{spec}(\Delta,M,g)\right)\leq\lambda_{1}(M)=\frac{(n-1)^2}{4}\alpha^{2},
 \end{eqnarray*}
 and two equalities can be achieved simultaneously when
 $M$ is isometric to $\mathbb{H}^{n}(-\alpha^2)$.}

\end{itemize}
}
\end{remark}

The nonlinear $\flat$-Laplace operator $\Delta_{\flat}$,
$1<\flat<\infty$, is a natural generalization of the linear
Laplacian $\Delta$, and, in a local coordinate system
$\{x^{1},x^{2},\cdots,x^{n}\}$ on $(M,g)$, is defined by
\begin{eqnarray*}
\Delta_{\flat}u=\mathrm{div}\left(\|\nabla u\|^{\flat-2}\nabla
u\right)=\frac{1}{\sqrt{\|g\|}}\frac{\partial}{\partial
x^{i}}\left(\sqrt{\|g\|}\cdot\|\nabla u\|^{\flat-2}\frac{\partial
u}{\partial x^{j}}\right).
\end{eqnarray*}
Domain monotonicity of eigenvalues with vanishing Dirichlet data
also holds for the first Dirichlet eigenvalue of $\Delta_{\flat}$
(see, e.g., \cite[Lemma 1.1]{dm2}) implies that
$\lambda_{1,\flat}^{D}(B(q,R))$ decreases as $R$ increases, where
$\lambda_{1,\flat}^{D}(B(q,R))$ is the first Dirichlet eigenvalue of
$\Delta_{\flat}$ on the geodesic ball $B(q,R)\subset M$. Besides,
using the variational principle, $\lambda_{1,\flat}^{D}(B(q,R))$ can
be characterized as follows
\begin{eqnarray} \label{4-7-1}
\lambda_{1,\flat}^{D}(B(q,R))=\inf\left\{\frac{\int_{B(q,R)}\|\nabla
u\|^{\flat}dv}{\int_{B(q,R)}|u|^{\flat}dv}\Bigg{|}u\neq0,~u\in
W^{1,\flat}_{0}(B(q,R))\right\},
\end{eqnarray}
where $W^{1,\flat}_{0}(B(q,R))$ is the completion of the set
$C^{\infty}_{0}(B(q,R))$ of smooth functions compactly supported on
$B(q,R)$ under the Sobolev norm
$\|\phi\|_{1,\flat}(B(q,R)):=\left[\int_{B(q,R)}\left(\|\nabla\phi\|^{\flat}+|\phi|^{\flat}\right)dv\right]^{1/\flat}$.
 One can
define the following
\begin{eqnarray*}
 \lambda_{1,\flat}(M):=\lim\limits_{R\rightarrow\infty}\lambda_{1,\flat}^{D}(B(q,R)),
  \end{eqnarray*}
which is independent of the choice of the point $q\in M$. Similar to
Schoen-Yau's question \cite[page 106]{sy} mentioned above, one can
naturally ask ``\emph{for a complete noncompact Riemannian
$n$-manifold $M$, $n\geq2$, when do we have
$\lambda_{1,\flat}(M)>0$?}". We have already obtained some
interesting results (see \cite{dm2,lmwz,MTZ}). For instance, Mao, Tu
and Zeng \cite[Lemma 2]{MTZ} proved that for an $n$-dimensional
($n\geq2$) Hadamard manifold whose sectional curvature satisfies
$K\leq-\alpha^{2}<0$, $\alpha>0$,
\begin{eqnarray*}
\lambda_{1,\flat}(M)\geq\left[\frac{(n-1)\cdot\alpha}{\flat}\right]^{\flat}>0
\end{eqnarray*}
holds, which generalizes Mckean's estimate mentioned above. Clearly,
 generally for a complete noncompact $n$-manifold $M$, $n\geq2$, one has
$\inf\left(\mathrm{spec}(\Delta_{\flat},M,g)\right)\leq\lambda_{1,\flat}(M)$.
Inspired by the experience that some estimates for $\lambda_{1}(M)$
can be extended to the case of $\lambda_{1,\flat}(M)$ mentioned
above, naturally one could ask ``\emph{whether the upper bound
estimate (\ref{4-5-1}) for
$\inf\left(\mathrm{spec}(\Delta,M,g)\right)$ can be extended
similarly to the case of
$\inf\left(\mathrm{spec}(\Delta_{\flat},M,g)\right)$ or not?}". In
fact, we can prove the following:

\begin{corollary} \label{corollary4-7}
Under assumptions of Theorem \ref{theorem4-5}, we have
\begin{eqnarray*}
\inf\left(\mathrm{spec}(\Delta_{\flat},M,g)\right)&\leq&\left(\frac{1}{\flat}\limsup\limits_{R\rightarrow\infty}\frac{\mathrm{vol}(\partial
B(q,R))}{\mathrm{vol}(B(q,R))}\right)^{\flat}\nonumber\\
&\leq&
\left(\frac{c_{8}(n,p)}{\flat\sqrt{n-1}}\cdot\sqrt{\limsup\limits_{R\rightarrow\infty}\overline{k_{-}}\left(p,q,0,R\right)}\right)^{\flat}.
\end{eqnarray*}
In particular, if $k_{-}\left(p,q,\alpha^{2}\right)$ is finite,
$\alpha>0$, then we
 have
 \begin{eqnarray*}
\inf\left(\mathrm{spec}(\Delta_{\flat},M,g)\right)\leq\left(\frac{c_{8}(n,p)\cdot\alpha}{\flat}\right)^{\flat}.
 \end{eqnarray*}
\end{corollary}

\begin{proof}
By (\ref{4-7-1}), it is easy to know
\begin{eqnarray*}
\inf\left(\mathrm{spec}(\Delta_{\flat},M,g)\right)\leq\lambda_{1,\flat}(M)\leq\frac{\int_{0}^{\infty}\left|u'(t)\right|^{\flat}\cdot\mathrm{vol}(\partial
B(q,t))dt}{\int_{0}^{\infty}|u(t)|^{\flat}\cdot\mathrm{vol}(\partial
B(q,t))dt}
\end{eqnarray*}
provided $u(t)\in W^{1,\flat}(M)$. Using a similar argument to that
in the proof of Theorem \ref{theorem4-5} and choosing
$u(t)=e^{-(c_{11}+\flat\epsilon)t/\flat}$, $t>0$, in the above
inequality, we can obtain
 \begin{eqnarray*}
 \inf\left(\mathrm{spec}(\Delta_{\flat},M,g)\right)&\leq&\frac{\int_{0}^{\infty}\left|\frac{c_{11}+\flat\epsilon}{\flat}\cdot
e^{-(c_{11}+\flat\epsilon)t/\flat}\right|^{\flat}\cdot\mathrm{vol}(\partial
B(q,t))dt}{\int_{0}^{\infty}\left|e^{-(c_{11}+\flat\epsilon)t/\flat}\right|^{\flat}\cdot\mathrm{vol}(\partial
B(q,t))dt}\\
&\leq& \left(\frac{c_{11}+\flat\epsilon}{\flat}\right)^{\flat},
 \end{eqnarray*}
 which, by letting $\epsilon\rightarrow0$, implies
\begin{eqnarray*}
 \inf\left(\mathrm{spec}(\Delta_{\flat},M,g)\right)\leq\left(\frac{c_{11}}{\flat}\right)^{\flat}=\left(\frac{1}{\flat}\limsup\limits_{R\rightarrow\infty}\frac{\mathrm{vol}(\partial
B(q,R))}{\mathrm{vol}(B(q,R))}\right)^{\flat}.
\end{eqnarray*}
Then the rest part of the proof is almost the same with that in the
proof of Theorem \ref{theorem4-5}.
\end{proof}

\begin{remark}
\rm{ As mentioned in (1) of Remark \ref{remark4-6}, since
$\lim_{p\rightarrow\infty}c_{8}(n,p)=n-1$, the second estimate in
Corollary \ref{corollary4-7} is sharp for the hyperbolic $n$-space
$\mathbb{H}^{n}(-\alpha^2)$. By using Corollary \ref{corollary4-7},
the fact
$\inf\left(\mathrm{spec}(\Delta_{\flat},M,g)\right)\leq\lambda_{1,\flat}(M)$,
and the Cheng-type eigenvalue comparison \cite[Theorem 3.2]{m2}, we
have:

\begin{itemize}

\item \emph{For a given complete noncompact $n$-manifold $(M,g)$, $n\geq2$, if
$\mathrm{Ric}(\cdot,\cdot)\geq-(n-1)\alpha^{2}g(\cdot,\cdot)$ for
some nonzero constant $\alpha$, then
 \begin{eqnarray*}
\inf\left(\mathrm{spec}(\Delta_{\flat},M,g)\right)\leq\lambda_{1,\flat}(M)\leq\left(\frac{(n-1)\cdot\alpha}{\flat}\right)^{\flat},
 \end{eqnarray*}
 and two equalities can be achieved simultaneously when
 $M$ is isometric to $\mathbb{H}^{n}(-\alpha^2)$.}

\end{itemize}
}
\end{remark}

\section{Compactness, a Cheeger-type estimate for the shortest closed geodesic, and a Buser-type isoperimetric inequality}
\renewcommand{\thesection}{\arabic{section}}
\renewcommand{\theequation}{\thesection.\arabic{equation}}
\setcounter{equation}{0} \setcounter{maintheorem}{0}

As we know that one can define $C^{k+\alpha}$-convergence of tensors
on a given manifold, where $k\geq0$ is an integer, $0<\alpha\leq1$,
and then the convergence concept of a collection of manifolds can be
given -- see, e.g., \cite[page 169]{pp} for details. Since there is
no significant difference between pointed and unpointed topologies
when manifolds in use are closed, in this section we will only focus
on complete or closed Riemannian manifolds. A collection of
Riemannian $n$-manifolds is said to be precompact in
$C^{k+\alpha}$-topology if any sequence in this collection has a
subsequence that is convergent in the $C^{k+\alpha}$-topology
(equivalently, if the $C^{k+\alpha}$-closure is compact).

By using the volume comparison conclusions proven in Section
\ref{VC}, we can get several compactness and finiteness results.
First, by applying Corollary \ref{corollary3-6} directly, if the
average of the Type-I Ricci curvature
$\overline{k_{-}}(\cdot,\cdot,\cdot,\cdot)$ is small enough, we have
the following compactness conclusion.

\begin{corollary} \label{corollary5-1}
Given a class of closed $n$-dimensional ($n\geq2$) Riemannain
manifolds $M$, $q\in M$, real numbers $p>\frac{n}{2}$, $D<\infty$,
and a function $\lambda(t)$ on $M$ satisfying the property
\textbf{P1},  we can find
$\epsilon:=\epsilon\left(n,p,q,\lambda(t),D\right)>0$ such that if
\begin{eqnarray*}
&&\quad\ell(q)\leq D, \qquad\qquad with~\ell(q)~defined~by~(\ref{key-def1}),\\
&&\overline{k_{-}}\left(p,q,\lambda(t),D\right)\leq\epsilon,
\end{eqnarray*}
 then $M\backslash
Cut(q)$ is precompact in the Gromov-Hausdorff topology.
\end{corollary}

We also would like to give another interesting application of volume
estimates obtained in Section \ref{VC}.

Applying the upper bound estimate for volume of normal tubes of a
given geodesic $\mathcal{C}$ (cf. Theorem \ref{theorem3-9}), if the
Type-II integral radial sectional curvature w.r.t. $\mathcal{C}$ was
assumed to be small enough, we can bound from below the length of
the shortest closed geodesic, which generalizes Cheeger's related
estimate shown in \cite{c1}.

\begin{theorem}  \label{theorem5-2}
Given a closed $n$-dimensional ($n\geq2$) Riemannain manifold $M$,
the shortest closed geodesic $\mathcal{C}$ on $M$, $p>n-1$, $w>0$,
$D<\infty$, and a function $\lambda(t)$ on $M$ satisfying the
property \textbf{P2},  we can find
$\epsilon:=\epsilon(n,p,\lambda(t),w,D)>0$ and
$\delta:=\delta(n,p,\lambda(t),w,D)>0$ such that if
\begin{eqnarray*}
&&\quad \ell(\mathcal{C})\leq D,\\
&& \quad \mathrm{vol}(M)\geq w,\\
&&
k^{\ast}_{+}(p,\mathcal{C},\lambda(t),D)\leq\epsilon(n,p,\lambda(t),w,D),
\end{eqnarray*}
with, similarly, $\ell(\mathcal{C}):=\sup_{q\in\mathcal{C}}\ell(q)$,
then the length $\mathbb{L}_\mathcal{C}$ of $\mathcal{C}$ satisfies
$\mathbb{L}_\mathcal{C}\geq\delta(n,p,\lambda(t),w,D)$.
\end{theorem}

\begin{proof}
Parameterize the closed geodesic $\mathcal{C}$ by arc-length as
$\zeta:[0,l]\rightarrow\mathcal{C}$. Consider a geodesic
$N=\zeta|_{(0,l)}$, which is actually obtained by deleting a single
point from $\mathcal{C}$. If $\ell(\mathcal{C})\leq D$, then the
normal tube $\mathcal{T}(N,D)$ w.r.t. $\mathcal{C}=N$ covers $M$
except a set of measure zero (i.e., the cut-locus of that single
point). Hence, one has
$\mathrm{vol}(M)=\mathrm{vol}\left(\mathcal{T}(N,D)\right)$ and
$k^{\ast}_{+}(p,\mathcal{C},\lambda(t),D)=k^{\ast}_{+}(p,N,\lambda(t),D)$.
Together with Theorem \ref{theorem3-9}, it follows that
\begin{eqnarray*}
w\leq\mathrm{vol}(M)=\mathrm{vol}\left(\mathcal{T}(N,D)\right)\leq\digamma(n,p,l,\lambda_{\inf},k_{+}^{\ast}(N),D)
\end{eqnarray*}
for some constant $\digamma$, where, similarly,
$k_{+}^{\ast}(N)=k^{\ast}_{+}(p,N,\lambda(t),D)=k^{\ast}_{+}(p,\mathcal{C},\lambda(t),D)$,
$\lambda_{\inf}=\inf_{(0,D)}\lambda(t)$. Moreover,
$\digamma(n,p,l,\lambda_{\inf},k_{+}^{\ast}(N),D)\rightarrow0$ as
$l, k_{+}^{\ast}(N)\rightarrow0$. Therefore, if
$$w\leq\digamma(n,p,l,\lambda_{\inf},k_{+}^{\ast}(N),D)$$ holds for
some fixed $w$ and $D$, then two nonnegative quantities $l$ or
$k_{+}^{\ast}(N)$ must be bounded from below by some positive
constant. Hence, if
$k_{+}^{\ast}(N)\leq\epsilon(n,p,\lambda(t),w,D)$ was assumed for
some small enough $\epsilon$, then there must exist some positive
constant $\delta$ such that
$$l>\delta(n,p,\lambda(t),w,D),$$
which completes the proof.
\end{proof}

Assume that $M$ is an $n$-dimensional ($n\geq2$) complete Riemannain
manifold, and, as before, denote by $B(q,R)$ the geodesic ball with
center $q\in M$ and radius $R$. The classical Buser's isoperimetric
inequality (see \cite{pb}) says:

\begin{itemize}

\item If the Ricci curvature of $M$ satisfies
$\mathrm{Ric}_{M}\geq(n-1)\lambda$ for some non-positive constant
$\lambda\leq0$, then there exists a positive constant
$c_{12}(n,\lambda,R)$, depending only on $n$, $\lambda$, $R$, such
that for any $q\in M$, a dividing hypersurface $\Gamma$ with
$\overline{\Gamma}$ embedded in $\overline{B(q,R)}$ and
$B(q,R)\setminus\Gamma=D_{1}\cup D_{2}$, we have
 \begin{eqnarray*}
\mathrm{vol}(\Gamma)\geq
c_{12}(n,\lambda,R)\cdot\min\{\mathrm{vol}(D_{1}),\mathrm{vol}(D_{2})\},
 \end{eqnarray*}
where $D_1$ and $D_2$ are two disjoint open sets contained in
$B(q,R)$.

\end{itemize}
Chavel \cite{ic2} extended the above conclusion for a more general
domain, i.e., a star-shaped domain $D$ with
$B\left(q,\frac{R}{2}\right)\subset D\subset B(q,R)$. Under a weaker
curvature assumption (i.e., if the integral Ricci curvature is
bounded from above), by applying the volume comparison \cite[Theorem
1.1]{pw1} obtained by Petersen and Wei, a generalized Buser-type
isoperimetric inequality has been attained by Paeng -- see
\cite[Theorem 1.2]{shp} for details. Here, using our volume estimate
(see Theorem \ref{theorem2-1}), we can get a more general Buser-type
isoperimetric inequality, which somehow improves the corresponding
conclusions in \cite{pb,shp}. In fact, we can prove:

\begin{theorem} \label{theorem5-3}
If $B(q,R)$ is a convex geodesic ball on a given complete Riemannian
$n$-manifold $M$, $n\geq2$, $k_{-}(p,q,\lambda(t),R)\leq K$ for some
nonnegative constant $K\geq0$, $p>\frac{n}{2}$, and  $\lambda(t)$ is
a continuous function on $M$ satisfying the property \textbf{P1},
then there exist positive constants $c_{13}(n,\lambda(t),R)$,
depending only on $n$, $\lambda(t)$, $R$, and
$c_{14}(n,p,\lambda(t),K,R)$, depending on $n$, $p$, $\lambda(t)$,
$K$, $R$, such that for a dividing hypersurface $\Gamma$ in $B(q,R)$
with $\overline{\Gamma}$ embedded in $\overline{B(q,R)}$ and
$B(q,R)\setminus\Gamma=D_{1}\cup D_{2}$, we have
\begin{eqnarray*}
\mathrm{vol}(\Gamma)\geq
c_{13}(n,\lambda(t),R)\cdot\min\{\mathrm{vol}(D_{1}),\mathrm{vol}(D_{2})\}-c_{14}(n,p,\lambda(t),K,R),
\end{eqnarray*}
where $D_1$, $D_2$ are two disjoint open sets contained in $B(q,R)$,
and for any fixed $\alpha\in(0,1)$, the precondition
 \begin{eqnarray} \label{5-3-assume}
\mathrm{vol}\left(D_{i}\cap
B(q,\frac{R}{2})\right)\leq\alpha\cdot\mathrm{vol}(D_i), \qquad
i=1~\mathrm{or}~2,
 \end{eqnarray}
 is satisfied.
\end{theorem}

\begin{proof}
we use a similar method to that in the proof of \cite[Theorem
1.1]{ic2}.

Without loss of generality, we assume that (\ref{5-3-assume}) was
satisfied with $i=1$, that is, $$\mathrm{vol}\left(D_{1}\cap
B(q,\frac{R}{2})\right)\leq\alpha\cdot\mathrm{vol}(D_1).$$ Fix
$t\in\left(0,\frac{R}{2}\right)$. For $x\in D_{1}\setminus Cut(q)$,
with, as before, $Cut(q)$ the cut locus of $q$, let
$\mathcal{C}_{xq}$ be the geodesic segment emanating from $x$ and
joining $q$, and let $x^{\ast}\in\mathcal{C}_{xq}$ be the first
point where $\mathcal{C}_{xq}$ intersects $\Gamma$. Clearly, if
$\mathcal{C}_{xq}\subset D_{1}$, then $x^{\ast}=q$. Define
\begin{eqnarray*}
&& A_{1}=\left\{x\in
D_{1}\backslash\left(Cut(q)\cup\overline{B(q,\frac{R}{2})}\right)\Bigg{|}x^{\ast}\in\left(\overline{B(q,t)}\right)^{c}\right\},\\
&&A_{2}=\left\{x\in
D_{1}\backslash\left(Cut(q)\cup\overline{B(q,\frac{R}{2})}\right)\Bigg{|}x^{\ast}\in\overline{B(q,t)}\right\},\\
&& A_{3}=\left(B(q,\frac{R}{2})\cap
B(q,t)\right)\cap\bigcup\limits_{x\in
A_2}\{\exp_{q}(\tau\theta)|t\leq\tau\leq s, \theta\in S_{q}^{n-1},
\|\theta\|=1,~x=\exp_{q}(s\theta)\},
\end{eqnarray*}
where $\left(\overline{B(q,t)}\right)^{c}$ is the complementary set
of $\overline{B(q,t)}$. Clearly, $A_1$, $A_2$, $A_3$ are subsets of
$D_{1}$. For a subset $S\subset T_{q}M$ of the tangent space
$T_{q}M$, by H\"{o}lder's inequality and Lemma \ref{lemma2-4}, we
have
 \begin{eqnarray} \label{5-3-1}
\int_{S}\int_{0}^{R}\psi(t,\xi)J^{n-1}(t,\xi)dtd\sigma&\leq&
\left(\int_{\mathbb{S}^{n-1}}\int_{0}^{R}\psi^{2p}J^{n-1}dtd\sigma\right)^{\frac{1}{2p}}\left(\int_{\mathbb{S}^{n-1}}\int_{0}^{R}J^{n-1}dtd\sigma\right)^{1-\frac{1}{2p}}
\nonumber\\
&\leq&
\left(c_{2}(n,p)\int_{\mathbb{S}^{n-1}}\int_{0}^{R}\rho^{p}J^{n-1}dtd\sigma\right)^{\frac{1}{2p}}\left(\mathrm{vol}(B(q,R))\right)^{1-\frac{1}{2p}}
\nonumber\\
&=&c_{2}^{\frac{1}{2p}}\cdot\left(\mathrm{vol}(B(q,R))\right)^{1-\frac{1}{2p}}\cdot\left(k_{-}(p,q,\lambda(t),R)\right)^{\frac{1}{2}},
 \end{eqnarray}
where $\psi(t,\xi)$ is defined as in Lemma \ref{lemma2-2},
$c_{2}(n,p)=\left(\frac{1}{n-1}-\frac{1}{2p-1}\right)^{-p}$ is the
constant given in Lemma \ref{lemma2-4}, and
$k_{-}(p,q,\lambda(t),R)$ is the Type-I integral radial Ricci
curvature w.r.t. $q$.

On the other hand, since
\begin{eqnarray*}
\frac{d}{ds}\left(\frac{J^{n-1}(s,\xi)}{f^{n-1}(s)}\right)\leq\psi(s,\xi)\cdot\frac{J^{n-1}(s,\xi)}{f^{n-1}(s)},
\end{eqnarray*}
with $f$ the solution to the ODE (\ref{ODE}), integrating from
$r_{1}\leq s$ to $s$ yields
\begin{eqnarray}  \label{5-3-2}
\frac{J^{n-1}(s,\xi)}{f^{n-1}(s)}-\frac{J^{n-1}(r_1,\xi)}{f^{n-1}(r_1)}\leq
\int_{r_1}^{s}\psi(t,\xi)\frac{J^{n-1}(t,\xi)}{f^{n-1}(t)}dt\leq\frac{1}{f^{n-1}(r_1)}\int_{r_1}^{s}\psi(t,\xi)J^{n-1}(t,\xi)dt,
\end{eqnarray}
where the last inequality holds because $\lambda(t)\leq0$ for $0\leq
t\leq s\leq R$, leading to the fact that $f(t)$ is increasing on
$[0,R)$. Hence,  from the above inequality, one can obtain, for
$r_{2}> r_{1}$, that
\begin{eqnarray*}
\int_{r_1}^{r_2}J^{n-1}(s,\xi)ds&\leq&\int_{r_1}^{r_2}\left(\frac{J^{n-1}(r_1,\xi)}{f^{n-1}(r_1)}+\frac{1}{f^{n-1}(r_1)}\int_{r_1}^{s}\psi(t,\xi)J^{n-1}(t,\xi)dt\right)f^{n-1}(s)ds\\
&\leq&\left(\frac{J^{n-1}(r_1,\xi)}{f^{n-1}(r_1)}+\frac{1}{f^{n-1}(r_1)}\int_{r_1}^{r_2}\psi(s,\xi)J^{n-1}(s,\xi)ds\right)\int_{r_1}^{r_2}f^{n-1}(s)ds\\
&\leq&\frac{\mathrm{vol}(\mathcal{B}(q^{-},r_{2}))-\mathrm{vol}(\mathcal{B}(q^{-},r_{1}))}{f^{n-1}(r_1)}\left(J^{n-1}(r_1,\xi)+\int_{r_1}^{r_2}\psi(s,\xi)J^{n-1}(s,\xi)ds\right),
\end{eqnarray*}
which implies
\begin{eqnarray} \label{5-3-3}
\frac{\int_{r_1}^{r_2}J^{n-1}(s,\xi)ds}{\mathrm{vol}(\mathcal{B}(q^{-},r_{2}))-\mathrm{vol}(\mathcal{B}(q^{-},r_{1}))}\leq
\frac{J^{n-1}(r_1,\xi)}{f^{n-1}(r_1)}+\frac{1}{f^{n-1}(r_1)}\int_{r_1}^{r_2}\psi(s,\xi)J^{n-1}(s,\xi)ds,
\end{eqnarray}
where, as before, $\mathcal{B}(q^{-},\cdot)$ denotes the geodesic
ball, with center $q^{-}$ and the prescribed radius, on the
spherically symmetric manifold
$M^{-}:=[0,\infty)\times_{f}\mathbb{S}^{n-1}$ with the base point
$q^{-}$. Therefore, for $0\leq r_0\leq l\leq r_{1}\leq r_{2}\leq R$,
by (\ref{5-3-2}) and (\ref{5-3-3}), one has
\begin{eqnarray*}
\frac{\int_{r_1}^{r_2}J^{n-1}(s,\xi)ds}{\mathrm{vol}(\mathcal{B}(q^{-},r_{2}))-\mathrm{vol}(\mathcal{B}(q^{-},r_{1}))}&\leq&
\frac{J^{n-1}(l,\xi)}{f^{n-1}(l)}+\frac{1}{f^{n-1}(l)}\int_{l}^{r_1}\psi(s,\xi)J^{n-1}(s,\xi)ds\\
&& \qquad
+\frac{1}{f^{n-1}(r_1)}\int_{r_1}^{r_2}\psi(s,\xi)J^{n-1}(s,\xi)ds\\
&\leq&
\frac{J^{n-1}(l,\xi)}{f^{n-1}(l)}+\frac{2}{f^{n-1}(l)}\int_{0}^{R}\psi(s,\xi)J^{n-1}(s,\xi)ds,
\end{eqnarray*}
which, together with (\ref{5-3-3}), implies
\begin{eqnarray} \label{5-3-4}
\frac{\int_{r_1}^{r_2}J^{n-1}(s,\xi)ds}{\mathrm{vol}(\mathcal{B}(q^{-},r_{2}))-\mathrm{vol}(\mathcal{B}(q^{-},r_{1}))}\leq
\frac{\int_{r_0}^{r_1}J^{n-1}(s,\xi)ds}{\mathrm{vol}(\mathcal{B}(q^{-},r_{1}))-\mathrm{vol}(\mathcal{B}(q^{-},r_{0}))}\nonumber\\
 \qquad
+\frac{2R}{\mathrm{vol}(\mathcal{B}(q^{-},r_{1}))-\mathrm{vol}(\mathcal{B}(q^{-},r_{0}))}\int_{0}^{R}\psi(s,\xi)J^{n-1}(s,\xi)ds.
\end{eqnarray}
Combining (\ref{5-3-1}) and (\ref{5-3-4}), we can obtain
\begin{eqnarray} \label{5-3-5}
&&\mathrm{vol}(A_2)\leq\frac{\mathrm{vol}(\mathcal{B}(q^{-},R))-\mathrm{vol}(\mathcal{B}(q^{-},R/2))}{\mathrm{vol}(\mathcal{B}(q^{-},R/2))-\mathrm{vol}(\mathcal{B}(q^{-},t))}
\cdot\Bigg{[}\mathrm{vol}(A_3)+2rc_{2}^{\frac{1}{2p}}\nonumber\\
&&\qquad\qquad\qquad
\cdot\left(\mathrm{vol}(B(q,R))\right)^{1-\frac{1}{2p}}\cdot
K^{\frac{1}{2}}\Bigg{]},
\end{eqnarray}
which, together with Theorem \ref{theorem2-1}, implies
\begin{eqnarray*}
&&\mathrm{vol}(A_2)\leq\frac{\mathrm{vol}(\mathcal{B}(q^{-},R))-\mathrm{vol}(\mathcal{B}(q^{-},R/2))}{\mathrm{vol}(\mathcal{B}(q^{-},R/2))-\mathrm{vol}(\mathcal{B}(q^{-},t))}
\cdot\Bigg{[}\mathrm{vol}(A_3)+2rc_{2}^{\frac{1}{2p}}\nonumber\\
&&\qquad\qquad\qquad \cdot\left(1+c(n,p,R)\cdot
K^\frac{1}{2}\right)^{2p-1}\left(\mathcal{B}_{n}(q^{-},R)\right)^{1-\frac{1}{2p}}\cdot
K^{\frac{1}{2}}\Bigg{]},
\end{eqnarray*}
where the constant $c(n,p,R)$ is given by (\ref{EXP-C}). Set
\begin{eqnarray}
&&\alpha_{1}:=\frac{\mathrm{vol}(\mathcal{B}(q^{-},R))-\mathrm{vol}(\mathcal{B}(q^{-},R/2))}{\mathrm{vol}(\mathcal{B}(q^{-},R/2))-\mathrm{vol}(\mathcal{B}(q^{-},t))},\nonumber\\
&&\alpha_{2}:=2c_{2}^{\frac{1}{2p}}\left(1+c(n,p,R)\cdot
K^\frac{1}{2}\right)^{2p-1}\left(\mathcal{B}_{n}(q^{-},R)\right)^{1-\frac{1}{2p}}\cdot
K^{\frac{1}{2}}\cdot\alpha_{1}.
\end{eqnarray}
Hence, one has
 \begin{eqnarray*}
\mathrm{vol}(A_2)\leq\alpha_{1}\mathrm{vol}(A_3)+r\alpha_{2}.
 \end{eqnarray*}
Since
 \begin{eqnarray*}
(1-\alpha)\mathrm{vol}(D_1)\leq\mathrm{vol}\left(D_{1}\setminus
B(q,\frac{R}{2})\right)=\mathrm{vol}(A_1)+\mathrm{vol}(A_2)
 \end{eqnarray*}
and
 \begin{eqnarray*}
\mathrm{vol}(A_3)\leq\mathrm{vol}\left(D_{1}\cap
B(q,\frac{R}{2})\right)\leq\alpha\mathrm{vol}(D_1),
 \end{eqnarray*}
we have
 \begin{eqnarray} \label{5-3-7}
\left(1-\alpha(1+\alpha_{1})\right)\mathrm{vol}(D_{1})\leq\mathrm{vol}(A_1)+\alpha_{2}R.
 \end{eqnarray}
Let $\{\exp_{q}(t\theta)\}\cap
A_{1}=\bigcup\limits_{\delta(\theta)}\left\{\exp_{q}(s\delta_{\theta})|s\in[\beta'_{\delta(\theta)},\gamma_{\delta(\theta)}],\|\delta_{\theta}\|=1\right\}$,
and moreover, set
\begin{eqnarray*}
\beta_{\delta(\theta)}=\left\{
\begin{array}{lll}
\beta'_{\delta(\theta)}, & \quad  \mathrm{if}~\beta'_{\delta(\theta)}>\frac{R}{2},\\
\left\|\exp^{-1}_{q}\left(\exp_{q}(\frac{R}{2}\theta)\right)^{\ast}\right\|,
& \quad \mathrm{if}~\beta'_{\delta(\theta)}=\frac{R}{2}.
\end{array}
\right.
\end{eqnarray*}
Let $\nu$ be the projection to $T_{q}M$ such that
$\nu(\exp_{q}(s\theta))=\theta$, and let $S$ be the subset
$\nu(A_1)\subset T_{q}M$. Using (\ref{5-3-3}) directly, one can
obtain, for $t<\frac{R}{2}\leq\beta_{\delta(\theta)}\leq
s\leq\gamma_{\delta(\theta)}\leq R$, that
\begin{eqnarray} \label{5-3-8}
&&\int_{S}\sum\limits_{\delta(\theta)}\int_{\beta_{\delta(\theta)}}^{\gamma_{\delta(\theta)}}J^{n-1}(s,\xi)dsd\theta\nonumber\\
&& \qquad \leq
\int_{S}\sum\limits_{\delta(\theta)}\left[\mathrm{vol}(\mathcal{B}(q^{-},\gamma_{\delta(\theta)}))-\mathrm{vol}(\mathcal{B}(q^{-},\beta_{\delta(\theta)}))\right]\nonumber\\
&& \qquad \qquad
\cdot\left(\frac{J^{n-1}(\beta_{\delta(\theta)},\xi)}{f^{n-1}(\beta_{\delta(\theta)})}+\frac{1}{f^{n-1}(\beta_{\delta(\theta)})}
\int_{\beta_{\delta(\theta)}}^{\gamma_{\delta(\theta)}}\psi(s,\xi)J^{n-1}(s,\xi)ds\right)d\theta\nonumber\\
&& \qquad \leq
\frac{\mathrm{vol}(\mathcal{B}(q^{-},R))-\mathrm{vol}(\mathcal{B}(q^{-},t))}{f^{n-1}(t)}\cdot\Bigg{[}\int_{S}\sum\limits_{\delta(\theta)}
J^{n-1}(\beta_{\delta(\theta)},\xi)d\theta\nonumber\\
&& \qquad\qquad +
\int_{S}\sum\limits_{\delta(\theta)}\int_{\beta_{\delta(\theta)}}^{\gamma_{\delta(\theta)}}\psi(s,\xi)J^{n-1}(s,\xi)dsd\theta\Bigg{]}.
\end{eqnarray}
On the other hand, similar to (\ref{5-3-1}), by H\"{o}lder's
inequality, Theorem \ref{theorem2-1} and Lemma \ref{lemma2-4}, one
has
\begin{eqnarray*}
&&\int_{S}\sum\limits_{\delta(\theta)}\int_{\beta_{\delta(\theta)}}^{\gamma_{\delta(\theta)}}\psi(s,\xi)J^{n-1}(s,\xi)dsd\theta
\\
&& \qquad  \leq
\left(\int_{S}\sum\limits_{\delta(\theta)}\int_{\beta_{\delta(\theta)}}^{\gamma_{\delta(\theta)}}\psi^{2p}(s,\xi)J^{n-1}(s,\xi)dsd\theta\right)^{\frac{1}{2p}}
 \cdot
\left(\int_{S}\sum\limits_{\delta(\theta)}\int_{\beta_{\delta(\theta)}}^{\gamma_{\delta(\theta)}}J^{n-1}(s,\xi)dsd\theta\right)^{1-\frac{1}{2p}}
\\
&& \qquad  \leq
\left(\int_{S}\int_{0}^{R}\psi^{2p}(s,\xi)J^{n-1}(s,\xi)dsd\theta\right)^{\frac{1}{2p}}
 \cdot
\left(\int_{S}\int_{0}^{R}J^{n-1}(s,\xi)dsd\theta\right)^{1-\frac{1}{2p}}
\\
&& \qquad  \leq
\left(\int_{S}\int_{0}^{R}\psi^{2p}(s,\xi)J^{n-1}(s,\xi)dsd\theta\right)^{\frac{1}{2p}}\cdot\left(\mathrm{vol}(B(q,R))\right)^{1-\frac{1}{2p}}
\\
&& \qquad  \leq
\left(\int_{\mathbb{S}^{n-1}}\int_{0}^{R}\psi^{2p}(s,\xi)J^{n-1}(s,\xi)dsd\theta\right)^{\frac{1}{2p}}
\left(1+c(n,p,R)\cdot
K^\frac{1}{2}\right)^{2p-1}\left(\mathcal{B}_{n}(q^{-},R)\right)^{1-\frac{1}{2p}}\\
&& \qquad  \leq c_{2}^{\frac{1}{2p}}\cdot
K^{\frac{1}{2}}\left(1+c(n,p,R)\cdot
K^\frac{1}{2}\right)^{2p-1}\left(\mathcal{B}_{n}(q^{-},R)\right)^{1-\frac{1}{2p}},
\end{eqnarray*}
which, together with (\ref{5-3-8}), implies
 \begin{eqnarray*}
\mathrm{vol}(A_1)&=&\int_{S}\sum\limits_{\delta(\theta)}\int_{\beta_{\delta(\theta)}}^{\gamma_{\delta(\theta)}}J^{n-1}(s,\xi)dsd\theta\nonumber\\
&\leq&
\frac{\mathrm{vol}(\mathcal{B}(q^{-},R))-\mathrm{vol}(\mathcal{B}(q^{-},t))}{f^{n-1}(t)}\cdot\Bigg{[}\int_{S}\sum\limits_{\delta(\theta)}
J^{n-1}(\beta_{\delta(\theta)},\xi)d\theta\\
&& \qquad\qquad + c_{2}^{\frac{1}{2p}}\cdot
K^{\frac{1}{2}}\left(1+c(n,p,R)\cdot
K^\frac{1}{2}\right)^{2p-1}\left(\mathcal{B}_{n}(q^{-},R)\right)^{1-\frac{1}{2p}}\Bigg{]}\\
&\leq&
\frac{\mathrm{vol}(\mathcal{B}(q^{-},R))-\mathrm{vol}(\mathcal{B}(q^{-},t))}{f^{n-1}(t)}\cdot\Bigg{[}\mathrm{vol}(\Gamma)\\
&& \qquad\qquad + c_{2}^{\frac{1}{2p}}\cdot
K^{\frac{1}{2}}\left(1+c(n,p,R)\cdot
K^\frac{1}{2}\right)^{2p-1}\left(\mathcal{B}_{n}(q^{-},R)\right)^{1-\frac{1}{2p}}\Bigg{]}.
 \end{eqnarray*}
Therefore, together with (\ref{5-3-7}), we have
 \begin{eqnarray} \label{5-3-9}
\mathrm{vol}(\Gamma)&\geq&\frac{f^{n-1}(t)}{\mathrm{vol}(\mathcal{B}(q^{-},r))-\mathrm{vol}(\mathcal{B}(q^{-},t))}\mathrm{vol}(A_{1})
\nonumber\\
&& \qquad - c_{2}^{\frac{1}{2p}}\cdot
K^{\frac{1}{2}}\left(1+c(n,p,R)\cdot
K^\frac{1}{2}\right)^{2p-1}\left(\mathcal{B}_{n}(q^{-},R)\right)^{1-\frac{1}{2p}} \nonumber\\
&\geq&
\frac{f^{n-1}(t)}{\mathrm{vol}(\mathcal{B}(q^{-},R))-\mathrm{vol}(\mathcal{B}(q^{-},t))}\mathrm{vol}(A_{1})-\frac{\alpha_{2}}{2\alpha_{1}}
\nonumber\\
&\geq&
\frac{f^{n-1}(t)}{\mathrm{vol}(\mathcal{B}(q^{-},R))-\mathrm{vol}(\mathcal{B}(q^{-},t))}\left[\left(1-\alpha(1+\alpha_{1})\right)\mathrm{vol}(D_{1})-\alpha_{2}R\right]-\frac{\alpha_{2}}{2\alpha_{1}}.
\qquad
 \end{eqnarray}
Clearly, one can choose
\begin{eqnarray*}
&&c_{13}(n,\lambda(t),R):=\frac{f^{n-1}(t)}{\mathrm{vol}(\mathcal{B}(q^{-},R))-\mathrm{vol}(\mathcal{B}(q^{-},t))}\cdot\left(1-\alpha(1+\alpha_{1})\right),\\
&&c_{14}(n,p,\lambda(t),K,R):=\frac{f^{n-1}(t)}{\mathrm{vol}(\mathcal{B}(q^{-},R))-\mathrm{vol}(\mathcal{B}(q^{-},t))}\cdot\alpha_{2}R+\frac{\alpha_{2}}{2\alpha_{1}},
\end{eqnarray*}
and then the conclusion of Theorem \ref{theorem5-3} follows
directly.
\end{proof}

\begin{remark}
\rm{It is easy to check that
$c_{14}(n,p,\lambda(t),K,R)\rightarrow0$ as $K\rightarrow0$. }
\end{remark}

\section{Open problems}
\renewcommand{\thesection}{\arabic{section}}
\renewcommand{\theequation}{\thesection.\arabic{equation}}
\setcounter{equation}{0} \setcounter{maintheorem}{0}

The fundamental solution of the heat equation is called \emph{the
heat kernel}, which can be bounded from both above and below in
terms of curvatures. More precisely, Debiard, Gaveau and Mazet
\cite{dge} gave an upper bound for the heat kernel on geodesic
balls, within the cut locus of the center, of manifolds with
sectional curvature bounded from above by some constant, while
Cheeger and Yau \cite{cy1} showed a lower bound estimate for the
heat kernel on geodesic balls of manifolds with Ricci curvature
bounded from below by some constant. Several years ago, Mao
\cite[Theorem 6.6]{m2} successfully extended these estimates to a
more general and interesting setting. In fact, he proved:\footnote{
We would like to mention one thing here, that is, the inequalities
(6.1) and (6.2) in \cite[Theorem 6.6]{m2} should have opposite
directions, i.e., they should be separately changed into
(\ref{6-1-1}), (\ref{6-1-2}) of Theorem \ref{theorem6-1} here.
However, the proof of \cite[Theorem 6.6]{m2} is almost correct
except that ``$\geq$" in (6.4), ``$\leq$" in (6.6) therein should
change directions.}

\begin{theorem} \label{theorem6-1}
Given a complete Riemannian $n$-manifold $M$, $n\geq2$, we can
obtain:

(1) if $M$ has a radial Ricci curvature lower bound
$(n-1)\lambda(t)$ w.r.t. some point $q\in M$, then, for
$r_{0}<\min\{\mathrm{inj}(q),l\}$, the inequality
\begin{eqnarray} \label{6-1-1}
H(q,y,t)\geq H_{-}\left(d_{M^{-}}\left(q^{-},z\right),t\right),
\end{eqnarray}
holds for all $(y,t)\in B(q,r_{0})\times(0,\infty)$ with
  $d_{M}(q,y)=d_{M^{-}}(q^{-},z)$ for any $z\in M^{-}$, where
  $\mathrm{inj}(q)$ is defined by (\ref{inj-R}),
  $M^{-}=:[0,l)\times_{f}\mathbb{S}^{n-1}$ with the base point $q^{-}$ and $f$ determined by
  (\ref{ODE}), $d_{M^{-}}$ and $d_{M}$ denote the distance functions
  on $M^{-}$, $M$ respectively. Moreover, the equality in
  (\ref{6-1-1}) holds at some $(y_{0},t_{0})\in
  B(q,r_{0})\times(0,\infty)$ if and only if $B(q,r_{0})$ is
  isometric to $\mathcal{B}(q^{-},r_{0})$;

(2) if $M$ has a radial sectional curvature upper bound $\lambda(t)$
w.r.t. some point $q\in M$, then, for $r_{0}<\min\{l(q),l\}$, the
inequality
\begin{eqnarray} \label{6-1-2}
H(q,y,t)\leq H_{+}\left(d_{M^{+}}\left(q^{+},z\right),t\right),
\end{eqnarray}
holds for all $(y,t)\in B(q,r_{0})\times(0,\infty)$ with
  $d_{M}(q,y)=d_{M^{+}}(q^{+},z)$ for any $z\in M^{+}$, where
  $l(q)$ is defined by (\ref{key-def1}),
  $M^{+}=:[0,l)\times_{f}\mathbb{S}^{n-1}$ with the base point $q^{+}$ and $f$ determined by
  (\ref{ODE}), $d_{M^{+}}$ and $d_{M}$ denote the distance functions
  on $M^{+}$, $M$ respectively. Moreover, the equality in
  (\ref{6-1-2}) holds at some $(y_{0},t_{0})\in
  B(q,r_{0})\times(0,\infty)$ if and only if $B(q,r_{0})$ is
  isometric to $\mathcal{B}(q^{+},r_{0})$.

(The boundary condition will either be Dirichlet or Neumann.)
\end{theorem}

As shown by \cite[Theorem 6.8]{m2},\footnote{ In fact, one can
easily find that (6.9) in \cite[Theorem 6.8]{m2} is a direct
consequence of (\ref{6-1-1}) here by choosing $y=q$. This implies
that although we have made minor typos in (6.1) and (6.2) of
\cite[Theorem 6.6]{m2}, but in the exhibition of the application of
Mao's heat kernel comparisons, correct forms have been used. Because
of this reason, after the formal publication of \cite{m2}, we did
not send an \emph{erratum} to the managing editor. We believe that
readers can easily find those minor typos by themselves.

 } a direct and important application of Theorem \ref{theorem6-1} is
that it can give an alternative proof of the Cheng-type eigenvalue
comparison conclusions \cite[Theorems 3.6 and 4.4]{fmi}.

The heat kernel upper bound estimate in \cite{dge} has been extended
to integral Ricci curvature by Gallot \cite[Theorem 6]{sg}, and the
heat kernel lower bound estimate in \cite{cy1} has been extended to
integral Ricci curvature by Dai and Wei \cite[Theorem 1.1]{dw1}. The
key point of those two extensions is that the error term can be
controlled by the integral Ricci curvature.

Inspired by the above facts, one might ask the following question.

\vspace{3mm}

\textbf{Problem 2}. \emph{Could Theorem \ref{theorem6-1} be extended
to the case of integral radial curvatures? Could we get heat kernel
estimates using the bounded integral radial curvatures assumption?}

\vspace{3mm}

Another question can also be issued naturally, that is,

\vspace{3mm}

\textbf{Problem 3}. \emph{Except conclusions shown in this paper,
what else can be extended to the setting of integral radial
curvatures? }

\vspace {5mm}

\section*{Acknowledgments}
\renewcommand{\thesection}{\arabic{section}}
\renewcommand{\theequation}{\thesection.\arabic{equation}}
\setcounter{equation}{0} \setcounter{maintheorem}{0}

This work was partially supported by the NSF of China (Grant No.
11401131), China Scholarship Council, the Fok Ying-Tung Education
Foundation (China), and Key Laboratory of Applied Mathematics of
Hubei Province (Hubei University). The author wants to thank the
Department of Mathematics, IST, University of Lisbon for its
hospitality during his visit from September 2018 to September 2019.

 \end{document}